%% file: KonjugationVonEinbettungen-v9.tex
\documentclass[a4paper,10pt]{article}
\usepackage{geometry}
\include{HeaderDSkJan2011}

\title{Field embeddings which are conjugate under a $p$-adic classical group}
\author{Daniel Skodlerack} 

\begin{document}
\maketitle
%*************************************************Seite 1**********************************************
\begin{abstract}
Let $(V,h)$ be a Hermitian space over a division algebra $D$ which is of index at most two over a non-Archimedean local field $k$ of residue characteristic not 2. Let $G$ be the unitary group defined by $h$ and let $\sigma$ be the adjoint involution. Suppose we are given two $\sigma$-invariant but not $\sigma$-fixed field extensions
$E_1$ and $E_2$ of $k$ in $\End_D(V)$ which are isomorphic under conjugation by an element $g$ of $G$ and suppose that there is a point $x$ in the Bruhat--Tits building of $G$ which is fixed by $E_1^{\times}$ and $E_2^{\times}$ in the reduced building of $\Aut_D(V)$. Then $E_1$ is conjugate to $E_2$ under an element of the stabilizer of $x$ in $G$ if $E_1$ and $E_2$ are conjugate under an element of the stabilizer of $x$ in $\Aut_D(V)$ and a weak extra condition. In addition in many cases the conjugation by $g$ from $E_1$ to $E_2$ can be realized as conjugation by an element of the stabilizer of $x$ in $G$. Further we give a concrete description of the canonical isomorphism from the set of $E_1$-times fixed points of the building of $G$ onto the building of the
centralizer of $E_1$ in $G$.  
\end{abstract}

\section{Introduction}
This article is about a Skolem--Noether kind of lemma in the framework of $p$-adic classical groups in the case of odd residue characteristic. For general linear groups these kinds of lemmas encode the invariants for analyzing the rigidity of certain irreducible representations on open compact subgroups of a $p$-adic group. These representations, called simple types, are used to construct and classify supercuspidal representations of the group of interest. The rigidity question for two simple types asks how they are related if they are contained in the same supercuspidal representation. For example the work of Broussous and Grabitz \cite{broussousGrabitz:00} is used in the above classification question for $\GL_m(D)$. The latter is mainly done by Secherre, Stevens and Broussous in the framework of Bushnell-Kutzko theory. In this article we consider $p$-adic classical groups. Stevens constructed all supercuspidal representations for the case where $D$ is a field \cite{stevens:08}. To understand the rigidity question for the case where $D$ has index two we need a new Skolem--Noether lemma, i.e. a classical version of part one of \cite{broussousGrabitz:00}.

To understand where the Skolem--Noether like proposition is involved let us give more details. 
Let $H$ be a $p$-adic group of the kind mentioned above. A simple type itself is constructed by a combinatorial algebraic object, called a simple stratum, which especially consists of a facet of the Bruhat--Tits building $\building(H)$, represented by a certain hereditary order $\mf{a}$ in an Azumaya algebra $A$ over a non-Archimedean local field $k,$ and a field extension $E|k$ in $A$ such that $E^{\times}$ normalizes $\mf{a}$. Note that $A$ always is the endomorphism ring of the vector space on which $H$ is defined. We call such a pair $(E,\mf{a})$ an embedding.
An approach to the rigidity question for simple types is to classify the $\mf{n}(\mf{a})\cap H$-conjugation classes of embeddings with same hereditary order, where $\mf{n}(\mf{a})$ is the normalizer of $\mf{a}$ in $A^{\times}$. In 
\cite{broussousGrabitz:00} the authors described these classes for the case that $H$ is  $A^{\times}$, i.e. some $\GL_m(D)$, using an equivalence relation on the set of embeddings which has the property that every equivalence class contains an embedding whose field is isomorphic to an unramified extension of $k$ in $D$. Moreover from their article it is easy to deduce the following proposition.

\begin{proposition}
If two embeddings $(E_i,\mf{a})$ of $A$ with $k$-algebra isomorphic fields and the same hereditary order $\mf{a}$ are equivalent then any $k$-algebra isomorphism between the two fields can be realized as conjugation by an element of $\mf{n}(\mf{a})$.
\end{proposition}
%****************************************Seite 2******************************************************

This article provides an analogous result for $p$-adic classical groups. Here we consider a $p$-adic unitary group $G$ defined by a signed Hermitian space $(V,h)$ over a central division algebra of finite index $d$ over a non-Archimedean local field $k$ of odd residue characteristic. This implies $d=1$ or $2$. Let $A$ be $\End_D(V)$. We consider embeddings $(E,\mf{a})$ which are invariant under the adjoint involution $\sigma$ of $h$, i.e. $\sigma(E)$ is equal to $E$ and $\sigma(\mf{a})$ is equal to $\mf{a}$. 
The main result of this note is the following:

\begin{theorem}\label{thmMainclassicalEmbeddings}
Suppose we are given two $\sigma$-invariant embeddings $(E_i,\mf{a})$ of $A$ and an element $g$ of $G$ such that $gE_1g^{-1}$ is equal to $E_2$.
\be 
\item[(i)] Assume that
\bi
\item[(a)] $\Gcd{f(E_1|k)}{d}|e(E_1|k)$ or
\item[(b)] $\frac{\dim_D(V)\Gcd{[E_1:k]}{d}}{[E_1:k]}$ is odd and both embeddings are equivalent.
\ei
Then there is a $g_1\in G\cap\mf{a}$ such that $g_1eg_1^{-1}$ is equal to $geg^{-1}$, for all $e\in E_1$. 
\item[(ii)]
In all other cases than in (a) and (b) if $\sigma|_{E_1}\neq \id_{E_1}$ and the non-trivial Galois element of the unramified field extension of degree two in $E_1|k$ can be extended to an automorphism of $E_1|k$ and both embeddings $(E_i,\mf{a}),\ i=1,2,$ are equivalent then there is an element $g_1$ of $G\cap\mf{a}$ such that $E_1^{g_1}$ is equal to $E_2$.   
\ee
\end{theorem}

The condition in (a) is not really strong, i.e. if $d=1$  there is no condition and if $d=2$ it says that $e(E_1|k)$ is even or $f(E_1|k)$ is odd. 
The extension condition in (ii) is necessary, see Proposition \ref{propExtCondNecessary}. For a study when it holds see section \ref{secExtensionProperty}. 

We decided to write a geometric proof for better readability. The fact that $E_1$ is conjugate to $E_2$ under $G$ is used to move the situation to the case that $E_1=E_2=E$ but with probably two different hereditary orders $\mf{a}$ and $\mf{a}'$. The latter are two facets in the Bruhat--Tits building $\building(G)$ stabilized by $E^\times$.  
We restrict the Broussous--Lemaire map (\cite[II.1.1]{broussousLemaire:02}) of reduced buildings
\[j_E:\ \building_{red}(\GL_D(V))^{E^\times}\ra\building_{red}((\GL_D(V))_{E})\]
to $\building(G)\cap\building_{red}(\GL_D(V))^{E^\times}$, which is mapped onto $\building(G_E)$. Here $(*)_{E}$ denotes the centralizer of $E$ in $(*)$.
In many cases it can be shown that $j_E$ maps the barycenters of $\mf{a}$ and $\mf{a}'$ to points of the same simplicial type in $\building(G_E)$.
Most arguments can be done on the level of $\GL_D(V)$. 

What remains then is the proof of part (ii). This step uses the construction of signed Hermitian modules. More precisely given a $k$-algebra homomorphism $\phi$ from a field extension $E|k$ to $A$ such that $\phi(E)$ is invariant under $\sigma$, then via pullback $\sigma$ defines an involution $\sigma'$ on $E$ and the vector space $V$ is a right $\tens{E}{k}{D}$-module. There is a procedure to construct a signed Hermitian form 
\[h^{\phi}:\ V\times V\ra \tens{E}{k}{D}\] 
on the module $V$, analogous to that used in \cite{broussousStevens:09}. For example $h^{\id_{E}}$ is used for a concrete description of the map $j_E|_{\building(G)\cap\building_{red}(\GL_D(V))^{E^\times}}$, see the appendix. In case of application of Theorem \ref{thmMainclassicalEmbeddings} the condition that $E_1$ is conjugate to $E_2$ under $G$ can be difficult to verify. For that the construction of the signed Hermitian modules can be very helpful as follows.

\begin{proposition}
 Let $\phi_1$ and $\phi_2$ be two $\sigma'$-$\sigma$-equivariant $k$-algebra homomorphisms from $E|k$ into $A$. Then, $\phi_1\circ\phi_2^{-1}$ can be described as a conjugation under $G$ if and only if $(V,h^{\phi_1})$ is isomorphic to $(V,h^{\phi_2})$ as signed Hermitian modules.
\end{proposition}

%********************************Seite 3 ***********************************************
The signed Hermitian modules can be analyzed using Witt decompositions, and lattice chains give the possibility to consider symmetric or skew-symmetric bilinear forms on residue spaces over finite fields. 

I thank Prof. S. Stevens for giving me the hint to this topic and the DFG for financing my position at University of Muenster for that project.

\section{The building of $\GL_D(V)$ and centralizers}

A good source for the preliminaries in this part is \cite{broussousLemaire:02}. Let $D$ be a skew field of finite index $d$ whose center is a non-Archimedean local field $k$, 
and let $V$ be a finite dimensional right $D$-vector space. we write
$\nu,o_D,\mathfrak{p}_D,\kappa_D$ for the valuation, the valuation ring, the valuation ideal and the residue field of $D$, respectively. We use similar notation for other local fields. The valuation is normalized such that $\nu(k^\times)$ is $\mathbb{Z}$.
We denote $\Aut_D(V)$ by $\tilde{G}$.

In these notes we work with Bruhat--Tits buildings in terms of lattice functions. Let us briefly repeat the basic concept here.
\begin{definition}
 An {\it $o_D$-lattice of $V$} is a finitely generated $o_D$-submodule of $V$ containing a $D$-basis of $V$. An {\it $o_D$-lattice function} is a family $\Lambda=(\Lambda(t))_{t\in \mathbb{R}}$ of $o_D$-lattices of $V$ such that for any real numbers $t<s$ we have
 \begin{enumerate}
  \item $\Lambda(t)\supseteq\Lambda(s)$,
  \item $\Lambda(t)=\bigcap_{r<t}\Lambda(r)$ and
  \item $\Lambda(t)\pi=\Lambda(t+\frac{1}{d})$ for any prime element $\pi$ of $D$.
 \end{enumerate}
The set of all $o_D$-lattice functions is denoted by $\Lattone{o_D}{V}$.
Let $s$ be real number. The {\it translation} of $\Lattone{o_D}{V}$ by $s$ is the map 
\[\Lambda\mapsto \Lambda+s,\ (\Lambda+s)(t):=\Lambda(t-s).\] 
A translation class of an $o_D$-lattice function is denoted by $[\Lambda]$.
A bijective $\subsetneq$-decreasing map from $\bbZ$ to the image of some $o_D$-lattice function $\Lambda$ is called an {\it $o_D$-lattice chain corresponding to $\Lambda$.}
The {\it square lattice function of $\Lambda$} is an $o_k$-lattice function of $\End_D(V)$ defined via
\[\mathfrak{g}_\Lambda(t):=\{a\in\End_D(V)\mid\ a(\Lambda(s))\subseteq\Lambda(s+t)\text{ for all }s\in\mathbb{R}\},\]
which by definition only depends on $[\Lambda]$. The hereditary order $\mathfrak{g}_\Lambda(0)$ only depends on an $o_D$-lattice chain corresponding to $\Lambda$.
In some arguments we need a right-limit of $\Lambda$ in $t$: We define 
\[\Lambda(t+):=\bigcup_{r>t}\Lambda(t).\]
\end{definition}

\begin{theorem}[\cite{broussousLemaire:02}]
There is an affine and $\tilde{G}$-equivariant bijection from $\Lattone{o_D}{V}$ to the Bruhat--Tits building $\building(\tilde{G})$ of $\tilde{G}$. Two such maps only differ by a translation. It induces the unique affine and $\tilde{G}$-equivariant
bijection $f$ from the set of all translation classes of $\Lattone{o_D}{V}$ to the reduced Bruhat--Tits building $\building_{red}(\tilde{G})$ of $\tilde{G}$. The facets of $\building_{red}(\tilde{G})$ correspond to the hereditary orders of $\End_D(V)$, more precisely the point 
$f([\Lambda])$ is a point of the facet $\{f([\Lambda'])|\  \im(\Lambda')=\im(\Lambda)\}$.
\end{theorem}

For a point $x$ of $\building_{red}(\tilde{G})$ we denote by $\mf{a}_x$ the hereditary order 
corresponding to $x$, i.e. $\mf{g}_\Lambda(0)$ if $f([\Lambda])=x$. We identify facets with 
hereditary orders. 
Let $E|k$ be a field extension in $\End_D(V)$. We denote the centralizer of $E$ in a subgroup $H$ of $\tilde{G}$ by $H_E$. 

\begin{theorem}[\cite{broussousLemaire:02},~II.1.1.]
There is a unique map $j_E$ from the set of $E^\times$-fixed points of $\building_{red}(\tilde{G})$, denoted by 
$\building_{red}(\tilde{G})^{E^\times}$, to the reduced building $\building_{red}(\tilde{G}_E)$ of $\tilde{G}_E$ such that
\[\mf{g}_x(t)\cap\End_{\tens{E}{k}{D}}(V)=\mf{g}_{j_E(x)}(t),\]
for all $x\in\building_{red}(\tilde{G})^{E^\times}$ and $t\in\mathbb{R}$.  
The map is
\be 
\item affine,
\item $\tilde{G}_E$-equivariant and
\item bijective. 
\ee
$j_E^{-1}$ is the unique map from $\building_{red}(\tilde{G}_E)$ to $\building_{red}(\tilde{G})$ which satisfies 1. and 2.
\end{theorem}
%*****************************************Seite 4******************************

\begin{remark}\label{remCentralizer}
 \be
\item The centralizer of $E$ in $\End_D(V)$ is $E$-algebra isomorphic to a ring $\End_\Delta(W)$ for some vector space 
$W$ over a skew-field $\Delta$. From Brauer theory it follows that the $\Delta$-dimension of $W$ is $\frac{m\Gcd{[E:k]}{d}}{[E:k]}$ 
and $\Delta$ has index $\frac{d}{\Gcd{[E:k]}{d}}$.  
\item Let us recall how the map $j_E$ is constructed. The map is induced by a $\tilde{G}_E$-equivariant affine map $j_E^1$ 
from $\building(\tilde{G})^{E^\times}$ to $\building(\tilde{G}_E)$, i.e. $j_E([\Lambda])=[j_E^1(\Lambda)]$. For the sake of simplicity and this article 
we only need the description for the case where $E|k$ is unramified and embeddable in $D|k$. Choose a maximal unramified field extension $L|k$ in $D$ and  
a uniformizer $\pi_D$ which normalizes $L$. Denote by $\Delta$ the centralizer in $D$ of the intermediate field between $L$ and $k$ of degree $[E:k]$.  Let $1^i,i\in\{1,\ldots,[E:k]\}$, be the idempotents of $\tens{E}{k}{L}$ and put $W:=1^1V$. The map $j_E^1$ is of the form 
\beq\label{eqjE1}
j_E^1(\Lambda)=\Theta\in\Lattone{o_\Delta}{W},\text{ s.t. }\Lambda(t)=\bigoplus_{i=0}^{[E:k]-1}\Theta\left(t-\frac{i}{d}\right)\pi_D^i,
\eeq
(\cite[II.3.1]{broussousLemaire:02}) and $\tilde{G}_E$ is identified with $\Aut_\Delta(W)$. 
\ee
\end{remark}

We now recall the concept of embeddings which is related to buildings of centralizers of $\tilde{G}$. We recommend \cite{broussousGrabitz:00} as a good introduction.
Let us denote by $E_D$ the intermediate field of $E|k$ which is unramified over $k$ and whose degree is the greatest common divisor of $d$ and the residue class degree $f(E|k)$ of $E|k$. This
 is exactly the greatest field in $E$ which can be embedded into a maximal unramified field extension $L|k$ of $D$. 
Let us recall the definition of an embedding.
 
 \begin{definition}
  \begin{enumerate}
   \item The {\it normalizer} of a hereditary order $\mf{a}$ of $\End_D(V)$ in $\tilde{G}$ is the set $\mf{n}(\mf{a})$ of all elements $g$ of $\tilde{G}$ for which $\mathfrak{a}$ is equal to $g\mf{a}g^{-1}$.
   \item An {\it embedding} is a pair $(E,\mathfrak{a})$ with a subfield $E$ of $\End_D(V)$ which extends $k$ and a hereditary order $\mathfrak{a}$ {\it normalized} by $E$, i.e. $ E^\times$ is a subset of $\mf{n}(\mf{a})$. 
   \item Two embeddings $(E_1,\mathfrak{a}_1)$ and $(E_2,\mathfrak{a}_2)$ are {\it equivalent} to each other if there is an element $g\in\tilde{G}$ such that $(g{E_1}_{D}g^{-1},g\mathfrak{a}_1g^{-1})$ is equal to $({E_2}_{D},\mathfrak{a}_2)$.
  \end{enumerate}
 \end{definition}

The importance of the equivalence of embeddings is described in the following theorem.   

\begin{proposition}[consequence of \cite{broussousGrabitz:00},~3.2]\label{propBroussousGrabitzNoetherThm}
\begin{enumerate}
\item Suppose we are given two equivalent embeddings \\ $(E_1,\mathfrak{a})$ and $(E_2,\mathfrak{a})$ of $\End_D(V)$ and a $k$-algebra isomorphism 
$\phi$ from $E_1$ to $E_2$. Then $\phi$ can be realized as conjugation by an element of $\mathfrak{n}(\mathfrak{a})$.
\item Two equivalent embeddings $(E,\mathfrak{a})$ and $(E,\mathfrak{a'})$ of $\End_D(V)$ are conjugate by an element $g$ of the centralizer $\tilde{G}_E$ of $E$ in $\tilde{G}$.  
\end{enumerate}
\end{proposition}
%*********************************Seite 5******************************

\begin{proof}
 \begin{enumerate}
  \item The finite field extension $E_1|k$ is primitive, because it is a tower of two primitive extensions where the first is unramified and thus separable. Fix a genrator $\alpha$ of $E_1$. We apply the Skolem--Noether Theorem to realize $\phi$ as a conjugation by an element $g$ of $\tilde{G}$. By \cite[3.2]{broussousGrabitz:00}, there 
is an element $g'$ of $\mathfrak{n}(\mathfrak{a})$ such that $g'\alpha g'^{-1}$ is equal to $g\alpha g^{-1}$. This proves 1.
\item There is an element $g$ of $\tilde{G}$ such that 
\[ gE_Dg^{-1}=E_D\text{ and }  g\mf{a'}g^{-1}=\mf{a}.\]
The statement follows from 1. applied to $(gEg^{-1},\mf{a})$, $(E,\mf{a})$ and $\phi$ defined by $\phi(x)=g^{-1}xg,\ x\in gEg^{-1}$.
 \end{enumerate}
\end{proof}

In many cases the embeddings $(E,\mathfrak{a})$ and $(E,\mathfrak{a'})$ are automatically equivalent if $\mathfrak{a}$ and $\mathfrak{a}'$ are conjugate under
$\tilde{G}$ even if $E_D$ is not $k$. In fact we have the following stronger lemma.

\begin{lemma}\label{lemGoingDownConugation}
Assume that $\frac{d}{\Gcd{d}{e(E|k)}}$ is prime to $f(E|k)$. 
If two elements $\Lambda,\ \Lambda'$ of $\Lattone{o_D}{V}$ are conjugate under an element $g\in\tilde{G}$, i.e. $g\Lambda=\Lambda'$, then there is an element $\tilde{g}$ of $\tilde{G}_{E_D}$ such that $\tilde{g}\Lambda=\Lambda'$.
\end{lemma}

\begin{proof}
 We consider the map $j_{E_D}^1$ from Remark \ref{remCentralizer}. We denote $j_{E_D}^1(\Lambda)$ by $\Theta$ which is an $o_\Delta$-lattice function where $\Delta$ is the centralizer of an embedding of $E_D$ in $D,$ and 
it is still an $o_E$-lattice function. More precisely there are elements $g_1$ and $g_2$ in $\tilde{G}_{E_D}$ such that 
\[g_1.\Theta=\Theta+\frac1{\Gcd{e(E|k)}{d}}\text{ and }g_2.\Theta=\Theta+\frac{[E_D:k]}d.\]
The assumption of the proposition implies that there is a product $g_3$ of powers of $g_1$ and $g_2$ such that $g_3.\Theta$ is equal to $\Theta+\frac1d$. 
We can conclude by the form of $j_{E_D}^1$ in Remark \ref{remCentralizer} that the $\kappa_k$-dimension of $\Theta(t)/\Theta(t+)$ is completely determined
by $\Lambda(t)/\Lambda(t+)$. The same is true for $\Lambda'$ and $\Theta':=j_{E_D}^1(\Lambda')$. The equality of the $\kappa_k$-dimension of $\Theta(t)/\Theta(t+)$ with that of 
$\Theta'(t)/\Theta'(t+)$, for all $t\in\bbR$, implies that there is an element $\tilde{g}$ of $\tilde{G}_E$ which satisfies  $\tilde{g}\Theta=\Theta'$ and thus $\tilde{g}\Lambda$ is equal to $\Lambda'$. 
\end{proof}

\begin{corollary}\label{corEmbTypepreserving}
Assume that $\frac{d}{\Gcd{d}{e(E|k)}}$ is prime to $f(E|k)$. 
Two embeddings $(E,\mf{a})$ and $(E,\mf{a}')$ are conjugate under $\tilde{G}_E$ if and only if $\mf{a}$ and $\mf{a}'$ are conjugate under $\tilde{G}$.
\end{corollary}

\begin{proof}
We only have to proof the "if" part. We apply firstly Lemma \ref{lemGoingDownConugation} on the barycenters of $\mf{a}$ and $\mf{a}'$ in $\building(\tilde{G})$ and secondly 
Proposition \ref{propBroussousGrabitzNoetherThm}.
\end{proof}

\begin{definition}
Let $B$ be a Euclidean building and fix a labeling of the vertices. The {\it (simplicial) type of $x$ in $B$} are the barycentric coordinates of $x$ with respect 
to the  vertexes of a chamber whose closure contains $x$. The barycentric coordinates do not depend on the chosen chamber. 
\end{definition}

\begin{remark}\label{remTypePreserving}
 \be
\item Two elements $x,y$ of $\building_{red}(\tilde{G})$ have the same type if and only if there is an element of $\SL_D(V)$ which maps $x$ to $y$. Indeed, $\SL_D(V)$ acts strongly transitive on the simplicial structure of $\building_{red}(\tilde{G})$.
\item  By 2. of Proposition \ref{propBroussousGrabitzNoetherThm} two facets of $\building_{red}(\tilde{G})$ give equivalent embeddings with respect to $E$ if and only if 
the $j_E$-images of the barycenters have the same type in $\building_{red}(\tilde{G}_E)$ up to a rotation of the Coxeter diagram of $\building_{red}(\tilde{G}_E)$.
\ee
\end{remark}
%*************************************Seite 6*******************************

Here we come to a proposition similar to Corollary \ref{corEmbTypepreserving}. The difference here is that we only consider type-preserving automorphisms of $\building_{red}(\tilde{G})$ and of $\building_{red}(\tilde{G}_E)$ 

\begin{proposition}\label{propTyppreserving}
Assume that $\frac{d}{\Gcd{d}{e(E|k)}}$ is prime to $f(E|k)$. 
Then two elements $x$ and $y$ of $\building_{red}(\tilde{G})^{E^\times}$ are of the same type in $\building_{red}(\tilde{G})$ if and only if $j_{E}(x)$ and $j_E(y)$ are of the same type in $\building_{red}(\tilde{G}_E).$
\end{proposition}

The barycenter of a facet $F$ is denoted by $\bary(F)$.
Although the statement is geometric the given proof is algebraic and uses the reduced norm. For more details about reduced norms we recommend \cite{reiner:03}. 
Write $A$ for $\End_D(V)$ and $A_E$ for the centralizer of $E$ in $A$. We write $\Nrd_{?|*}$ for the reduced norm and $\Nm_{?|*}$ if $?|*$ is a field extension. 
We write $\nu_E$ for the normalized valuation of $E$. In the proof we use the well-known fact that $g\in \tilde{G}$ is type-preserving if and only if 
$\nu(\Nrd_{A|k}(g))$ is a multiple of $m$, and we use the tower law: 
\[\Nrd_{A|k}(g)=\Nm_{E|k}(\Nrd_{A_E|E}(g)),\]
for $g\in B$.

\begin{proof}
The proof is motivated by the proof of \cite[3.2]{broussousGrabitz:00} (see Proposition \ref{propBroussousGrabitzNoetherThm} above) which also uses results of \cite{grabitz:99} on good continuations of hereditary orders, a terminology which we do not introduce here. The ``if'' part is trivial. 
\be
\item Case $\Gcd{f(E|k)}{d}=1$:
The important part is the ``only if'' one. The embeddings $(E,\mf{a}_x)$ and $(E,\mf{a}_y)$ are equivalent because of the condition on 
$f(E|k)$ and therefore there is an element $g_1$ of $\tilde{G}_E$ which conjugates $\mf{a}_x$ to $\mf{a}_y$. By assumption there is also an element $g$ of $\SL_D(V)$ which conjugates 
$y$ to $x$, i.e. $gg_1$ is an element of $\mf{n}(\mf{a}_x)$. Using the condition on $f(E|k)$ again there is an element $z$ of $\mf{n}(\mf{a}_{j_E(x)})$, the normalizer of $\mf{a}_{j_E(x)}$ in $A_E^\times$,  such that $\mf{n}(\mf{a}_x)$ is equal to $\mf{a}_x^\times\langle z\rangle$, see \cite[2.2]{grabitz:99}.  Thus there is an element $g_2$ of $\mf{n}(\mf{a}_{j_E(x)})$ such that $gg_1g_2$ lies in $\mf{a}^\times_x$, especially $\Nrd_{A|k}(g_1g_2)$ is a unit of $o_k$ and thus $\Nrd_{A_E|E}(g_1g_2)$ is a unit of $o_E$. This implies that $g_1g_2$ is type-preserving on $\building_{red}(\tilde{G}_E)$. In other words there is an element $g_3$ of $\tilde{G}_E$ with reduced norm one such that $g_3\mf{a}_{j_E(x)}g_3^{-1}$, which is $\mf{a}_{j_E(g_3.x)}$, is equal to $\mf{a}_{j_E(y)}$. 
We want to show $g_3.x=y.$ The condition on $f(E|k)$ to be prime to $d$ and \cite[2.2]{grabitz:99}
imply that there is only one facet $F$ of $\building_{red}(\tilde{G})^{E^\times}$ which satisfies 
that $j_E(\bary(F))$ is a point of $\mf{a}_{j_E(x)}$ (open facet). Thus $\mf{a}_y$ is equal to $\mf{a}_{g_3x}$, and thus $y=g_3.x$ because $g_3$ is type-preserving on $\building_{red}(\tilde{G})$.
\item Case $\frac{d}{\Gcd{d}{e(E|k)}}$ is prime to $f(E|k)$: Let $\Lambda$ and $\Lambda'$ be two lattice functions such that 
$x$ corresponds to $[\Lambda]$ and $y$ to $[\Lambda']$ and such that there is an element $g$ of $\SL_D(V)$ which satisfies $g\Lambda(t)=\Lambda'(t)$, for all $t\in\bbR$. Here we have used Remark \ref{remTypePreserving}(1). There is an element $g'\in\tilde{G}_{E_D}$ such that $g'\Lambda=\Lambda'$ by Lemma \ref{lemGoingDownConugation}.  
Thus $g'^{-1}g$ is an element of $\mf{a}_x^\times$ and we deduce that $\Nrd_{A|k}(g'^{-1})=\Nrd_{A|k}(g'^{-1}g)$ is an element of $o_k^{\times}$ and therefore $\Nrd_{A_{E_D}|E_D}(g'^{-1})\in o_{E_D}^{\times}$. It follows that $j_{E_D}(x)$ and $j_{E_D}(y)$ have the same simplicial type in $\building_{red}(\tilde{G}_{E_D})$ and the first case finishes the proof.  
\ee
\end{proof}

\begin{lemma}\label{lemTypePreservingembeddingType}
Suppose that there is an isometric simplicial group action on $\building_{red}(\tilde{G})$ by a cyclic group $T$ of order two. Assume that the generator $t$ of $T$ induces a reflection on the Coxeter diagram. Let $x$ and $y$ be two fixed points of $T$ in $\building_{red}(\tilde{G})$ with same cyclic barycentric coordinates up to rotation of the Coxeter diagram, i.e. there is an element $g$ of $\tilde{G}$ which sends $x$ to $y$.  Then we have:
\be
\item The points $x$ and $y$ have the same type in $\building_{red}(\tilde{G})$ if $\dim_DV$ is odd.
\item There are at most two different possible types for $y$ in $\building_{red}(\tilde{G})$ if $\dim_DV$ is even. 
\ee
\end{lemma}

%********************************Seite 7**************************************

\begin{proof} 
Fix a labeling of the vertices by $\mathbb{Z}/m\mathbb{Z}$, and let $(\alpha_i)_{i\in\Z/m\Z}$ be the barycentric coordinates of $x$. 
We define $H_1$ to be the subset of $\mathbb{Z}/m\mathbb{Z}$ of all $j$, such that $\alpha_{i+j}$ is equal to $\alpha_{t(i)+j}$ or equivalently to $\alpha_{i-j}$ 
for all $i\in\mathbb{Z}/m\mathbb{Z}$. For the last step we used that $t$ is a reflection of the Coxeter diagram and that $(\alpha_i)_{i\in\Z/m\Z}$ is invariant under $t$:
\[\alpha_{t(i)+j}=\alpha_{t(i-j)}=\alpha_{i-j}.\]
The set $H_1$ is a subgroup of $\mathbb{Z}/m\mathbb{Z}$. We further define $H_2$ to be the subgroup of $\mathbb{Z}/m\mathbb{Z}$ of all $j$
such that $\alpha_{i+j}$ is equal to $\alpha_{i}$ for all $i\in\mathbb{Z}/m\mathbb{Z}$. The group $H_1/H_2$ is cyclic, because $\mathbb{Z}/m\mathbb{Z}$ is, and 
every element of $H_1/H_2$ has order $1$ or $2$. Thus $H_1/H_2$ has at most two elements this proves 2. The groups $H_1$ and $H_2$ equal if $m$ is odd because 
the order of $H_1/H_2$ divides $m$. This proves 1.
\end{proof}

\section{The Euclidean building of a classical group}
From this section on, we consider the following situation. 
Let the residue characteristic of $k$ be odd and $\rho$ be an involution on $D,$ i.e. a bijective ring homomorphism from $D$ to $D^{op}$ of order two. The existence of the involution implies that the index is either 
one or two. We fix an element $\epsilon\in\{1,-1\}$ and a non-degenerate $\epsilon$-Hermitian form $h$ on $V$ with adjoint involution $\sigma$, i.e. $h$ is a non-degenerate bi-additive map from $V\times V$ to $D$ such that 
\[h(va,wb)=\rho(a)\epsilon\rho(h(w,v))b,\]
for all $v,w\in V$ and all $a,b\in D.$ We denote the fixed field of $k$ under $\rho$ by $k_0$.
The group of main interest in this article is the unitary group of $h$, i.e. 
\[G:=\U(h):=\U(\sigma):=\{g\in\tilde{G}\mid g\sigma(g)=\id_V\}.\]

We repeat the description of the Euclidean building $\building(G)$ of $G$ in terms of lattice functions. 
Let us recall that the {\it dual} $M^{\#,h}$ of an $o_D$-lattice  $M$ is defined via 
\[M^{\#,h}:=\{v\in V\mid\ h(v,M)\subseteq \mf{p}_D\}.\]
The {\it dual} $\Lambda^{\#,h}$ of an $o_D$-lattice function $\Lambda$ of $V$ is defined via
\[\Lambda^{\#,h}(t):=(\Lambda((-t)+))^{\#,h}.\]
We call $\Lambda$ {\it sef-dual} with respect to $h$ if it is equal to its dual. The set of self-dual lattice functions inherits an affine structure from $\Lattone{o_D}{V}$.
 
\begin{theorem}[\cite{bruhatTitsIV:87},~2.12, \cite{broussousStevens:09},~4.2,~3.3]
There is a unique $G$-equivariant affine bijection from $B(G)$ to the set of self-dual $o_D$-lattice functions of
$V$. 
\end{theorem}

In this article we always consider points of $\building(G)$ as self-dual $o_D$-lattice functions. 

\section{Extending Hermitian forms}\label{secExtHermForms}
Suppose we are given a finite field extension $E|k$ and a field automorphism $\sigma'$ of order at most two on $E$ which 
extends $\rho|_k.$ We consider the $E$-algebra $\tens{E}{k}{D}$ together with the involution $\tens{\sigma'}{k}{\rho}$, 
and in the manner of Broussous and Stevens given in \cite{broussousStevens:09} we 
fix a map $\lambda$ from $E$ to $k$ which 
\begin{equation}\label{assLambda}
\text{is non-zero, $k$-linear and $\sigma'$-$\sigma$-equivariant, and which satisfies }
\mf{p}_{E_0}=\{e\in E_0\mid\ \lambda(e o_{E_0})\subseteq\mf{p}_{k_0}\},
\end{equation}
where $E_0$ is the fixed field of $\sigma'$ in $E$. 
 We extend $\lambda$ to a map $\tilde{\lambda}$ from $\tens{E}{k}{D}$ to $D$ via
\[\tilde{\lambda}(\tens{x}{k}{y}):=\lambda(x)y.\]
%********************************Seite 8**************************************

Let us now choose a suitable $E$-structure on $V$ and assume that there is a $\sigma'$-$\sigma$-equivariant $k$-algebra homomorphism $\phi$ from $E$ to $\End_D(V)$.

We have that $V$ is in a canonical way a right-$\tens{E}{k}{D}$-module using $\phi$. We now define a Hermitian structure on that module. 

\begin{proposition}\label{propExtendHermForms}
There is a unique bi-additive map $h^{\phi}$ from $V\times V$ to $\tens{E}{k}{D}$ such that 
\begin{enumerate}
 \item $h^{\phi}(va,wb)=(\tens{\sigma'}{k}{\rho})(a) h^{\phi}(v,w)b$ and $(\tens{\sigma'}{k}{\rho})(h^{\phi}(v,w))=\epsilon h^{\phi}(w,v)$, for all $a,b\in\tens{E}{k}{D}$ and $v,w\in V$, and
 \item $\tilde{\lambda}\circ h^\phi=h.$
 \end{enumerate}
Assume we are given a second $\sigma'$-$\sigma$-equivariant $k$-algebra homomorphism $\phi'$ from $E$ into $\End_D(V)$. Then, $\phi\circ\phi'^{-1}$ 
can be described as a conjugation by an element of $G$ if and only if $(V,h^{\phi})$ is isomorphic to $(V,h^{\phi'})$ as signed Hermitian modules.
\end{proposition}

For the proof see \cite[5.2]{broussousStevens:09}. The proof also is valid if $\tens{E}{k}{D}$ is not a skew-field, because the theory of semisimple modules also implies 
that $\Hom_{\tens{E}{k}{D}}(V,\tens{E}{k}{D})$ and $\Hom_D(V,D)$ have the same $k$-dimension. The second part is trivial and can be left to the reader. 

For later purposes, we want to show that the passage from $h$ to $h^\phi$ respects the duality for lattice functions if $\tens{E}{k}{D}$ is a skew-field,
(see \cite[5.5]{broussousStevens:09} for $D=k$). So, let us assume for the rest of this section that $\tens{E}{k}{D}$ is a skew-field  and denote it by $\Delta$.
Recall that this is equivalent to $[E:k]$ being odd.  
To study $o_\Delta$ we fix prime elements $\pi_E$ and  $\pi_D$ of $E$ and $D$, respectively. 

\begin{lemma}\label{lemValONTensProd}
We abbreviate $e:=e(E|k)$ and $f:=f(E|k)$, and we extend the valuation $\nu$ on $k$ to a valuation, also denoted by $\nu$, on 
$o_{\Delta}$ and $\nu(t)$ it is equal to 
\[(\tens{\nu}{k}{\nu})(t):=\sup\{\inf\{\nu(e_i)+\nu(d_i)\mid\ i\in I\}\mid
 \sum_{i\in I}\tens{e_i}{k}{d_i}=t\},\]
for all $t\in\Delta$. 
Further we have:
\be
\item[(i)] The element \[\pi_{\Delta}:=\tens{\pi_E^{\frac{1-e}{2}}}{k}{\pi_D}\]
is a prime element of $\Delta.$ 
\item[(ii)] The valuation ring of $\Delta$ is 
\[\tens{o_E}{o_k}{o_D}+\tens{\mf{p}_E^{\frac{1-e}{2}}}{o_k}{\mf{p}_D}.\]
\item[(iii)] The valuation ideal of $\Delta$ is 
\[\tens{\mf{p}_E}{o_k}{o_D}+\tens{\mf{p}_E^{\frac{1-e}{2}}}{o_k}{\mf{p}_D}.\]
\ee
\end{lemma}

\begin{proof}
We only need to consider the case $D\neq k$ because the other is trivial. 
We choose elements $a_1$, $a_2$ of $o_D$ and $b_i$, $1\leq i\leq f$, of $o_E$ such that $(\bar{a}_1,\bar{a}_2)$ is a $\kappa_k$-basis of $\kappa_D$ and 
$(\bar{b}_1,\ldots,\bar{b}_f)$ a $\kappa_k$-basis of $\kappa_E$. 
Because of $f$ is odd, the tuple of products
\[(\bar{a}_i\bar{b}_j)_{i,j}\]
is a $\kappa_k$-basis of $\kappa_{\Delta}$.
The element $\pi_{\Delta}$ is a prime element of $\Delta$, because
\[\nu(\pi_{\Delta})=\frac{1-e}{2e}+\frac{1}{2}=\frac{1}{2e},\]
and the tuple
\[(\pi_{\Delta}^la_ib_j)_{0\leq l\leq 2e-1, 1\leq i\leq 2, 1\leq j\leq f}\]
is a $\nu$-splitting basis of $o_{\Delta}|o_k$. From this follows easily assertion (ii), and we get (iii) from 
\beq\label{eqNewCondOnLambda}
\pi_{\Delta}(\tens{o_E}{o_k}{o_D})=\tens{\mf{p}_E^{\frac{1-e}{2}}}{o_k}{\mf{p}_D}
\eeq
and 
\[\pi_{\Delta}(\tens{\mf{p}_E^{\frac{1-e}{2}}}{o_k}{\mf{p}_D})=\pi_k\tens{\mf{p}_E^{1-e}}{o_k}{o_D}=\tens{\mf{p}_E}{o_k}{o_D}.\]
By the Chinese reminder theorem the valuations 
\[\nu(\tens{\pi_E^{l_1}}{k}{\pi_D^{l_2}}),\ l_1\in\{0,\ldots,e-1\},\ l_2\in\{0,1\}\]
form a system of representatives of $\mathbb{Z}/(2e\mathbb{Z})$. And thus by 
\cite[(21)]{bruhatTitsIII:84} the valuation $\nu$ on $\Delta$ is as in the statement of the lemma.
\end{proof}

%**********************************Seite 9 Mitte*************************8

\begin{proposition}[Analogous to  \cite{broussousStevens:09},~5.5]\label{prop5_5brousStevForDneqkSkewField}
Let $\phi$ be a $k$-algebra monomorphism from $E$ to $\End_D(V)$. Let $M$ be an $o_D$-lattice of $V$ and assume $M$ to be an
$o_E$-lattice via $\phi$, i.e. an $o_\Delta$-lattice. We then have 
\[M^{\#,h}=M^{\#,h^{\phi}}.\]
\end{proposition}

\begin{proof}
First we show
\[\mf{p}_{\Delta}=\{x\in\Delta\mid\ \tilde{\lambda}(xo_{\Delta})\subseteq \mf{p}_D\}.\]

The inclusion $\subseteq$ follows from statement [(iii)] of Lemma \ref{lemValONTensProd},
and the equality (\ref{assLambda}), more precisely:
\[\tilde{\lambda}(\pi_\Delta)=\lambda(o_E)\mf{p}_D+\lambda(\mf{p}_E^{\frac{1-e}{2}})\mf{p}_D\subseteq \mf{p}_D.\]
The other inclusion $\supseteq$ follows because the $o_{\Delta}$-module on the right hand side does not contain $1_{\Delta}$ by (\ref{assLambda}). 

Now we repeat the argument of \cite[5.5]{broussousStevens:09}. 
\begin{eqnarray*}
M^{\#,h}&=& \{v\in V\mid\ h(v,M)\subseteq \mf{p}_D\}\\
&=& \{v\in V\mid\ \tilde{\lambda}(h^\phi(v,M))\subseteq \mf{p}_D\}\\
&=& \{v\in V\mid\ h^\phi(v,M)\subseteq \mf{p}_{\Delta}\}\\
&=& M^{\#,h^{\phi}}.
\end{eqnarray*}
\end{proof}

\section{Skolem--Noether for $U(h)$ and conjugate embeddings}
Here we analyze when two points of $\building(G)$ have the same type.  Afterwards we prove Theorem \ref{thmMainclassicalEmbeddings}. 

\begin{proposition}[\cite{skodlerack:10},~3.27]\label{propJEClassocalGroup}
Let $E|k$ be a $\sigma$-invariant field extension in $\End_D(V)$. Then the group $G_E$ is a unitary group, more precisely it coincides with 
$\U(\sigma|_ {\End_{\tens{E}{k}{D}}(V)})$. The map $j_E$ is $()^\#$-equivariant, no matter which signed Hermitian form we choose for $G_E$, and 
$j_E(\building(G)\cap \building_{red}(\tilde{G})^{E^\times})$ is equal to $\building(G_E)$.
\end{proposition}

In the appendix we give a concrete description of $j_E|_{\building(G)\cap \building_{red}(\tilde{G})^{E^\times}}$.
The next proposition gives a criteria for two points in $\building(G)$ to have the same type. 

\begin{proposition}\label{propSimplicialTypeInBGandBredtildeG}
Suppose we are given two elements $x$ and $y$ of $\building(G)$ with corresponding self-dual lattice functions $\Lambda_x$ and $\Lambda_y$, respectively.
Then the following statements are equivalent. 
\be
\item The points $x$ and $y$ have the same type in $\building(G)$.
\item The points $x$ and $y$ have the same type in $\building_{red}(\tilde{G})$.
\item For every real number $t$ the quotient $\Lambda_x(t)/\Lambda_x(t+)$ has the same $\kappa_D$-dimension as $\Lambda_y(t)/\Lambda_y(t+)$.
\item There is an element $g$ of $G$ such that $g.x$ is equal to $y$. 
\ee
\end{proposition}

\begin{remark}
There is only one case where $B(G)$ is not $B_{red}(G)$ it is the case of the isotropic orthogonal group over $k$ of rank one.  In this case we just take assertion 4. as definition for two points to have the same type
in $B(G)$.
\end{remark}

%*********************Seite 10**************************************

\begin{proof}
\bi
\item The equivalence of 1. and 4. is general theory of Bruhat--Tits buildings which can be found in \cite{corvallisTits:79}. Statement 2. follows from 4. because the reduced norm of an element $g$ of $G$ is an element of $o_k^\times$, i.e. $g$ acts type-preservingly on $\building_{red}(\tilde{G})$. 
\item We now prove that 3. follows from 2. The group $\SL_D(V)$ acts transitively on points of the same type in $\building_{red}(\tilde{G})$, i.e. there is an element $g$ of $\SL_D(V)$ which sends $x$ to $y$, i.e. in terms of lattice functions there is a real number $s$ such that $g.\Lambda_x$ is equal to $\Lambda_y+s$. Without loss of generality we can assume that $s$ is non-positive, because otherwise we can interchange $x$ and $y$. The self-duality of both lattice functions implies
\[\sigma(g)^{-1}.\Lambda_x=\sigma(g)^{-1}.\Lambda_x^{\#}=(g.\Lambda_x)^{\#}=\Lambda_y^{\#}-s=\Lambda_y-s.\]
Thus $g_1:=\sigma(g)g$ is an element of $\mf{g}_{\Lambda_x}(-2s)\cap\SL_D(V)$. The reduced norm of an element of the radical  $\mf{g}_{\Lambda_x}(0+)$ of 
$\mf{g}_{\Lambda_x}(0)$ is an element of $\mf{p}_k$, i.e. cannot be 1. In particular, $s$ vanishes and we obtain the third assertion.
\item The difficult part of the proposition is that 4. follows from 3. 
We consider the $o_D$-lattice chain $L_*$ corresponding to $\Lambda_x$ such that $L_0=\Lambda_x(0)$ and analogously $L'_*$ for $\Lambda_y$. 
There is a Witt decomposition $\{W_i\mid\ i\in I\}$ of $V$ with respect to $h$ which splits both lattice chains. Without loss of generality we can assume that the anisotropic part $W_0$ of the
Witt decomposition is trivial, because $L$ and $L'$ are equal on $W_0$ by \cite[2.9]{bruhatTitsIV:87}. Let $r$ be the period of $L.$ We choose a decomposition of $I$ into two disjoint sets $I^+$ and $I^-$ such that \[\sigma(I^+)=I^-.\] Let us recall that $\sigma(i)$ is defined to be $i'$ if $\sigma$ sends the projection onto $W_i$ to the projection onto $W_{i'}$.
 Further we define
\[W^+:=\oplus_{i\in I^+}W_i,\ W^-:=\oplus_{i\in I^-}W_i\]
and
\[L^+:=L\cap W^+ ,\ L^-:=L\cap W^-.\]
Let $\mu(L,j)$ be the set of indices $i\in I$ for which $W_i\cap L_j$ differs from $W_i\cap L_{j+1}$.
Analogously we define $\mu(L^+,j)$ and $\mu(L^-,j)$. Caution: one of the latter sets can be empty. 

\textbf{Case 1} ($L_0^{\#}=L_0$ and $2|r$): We choose, for $0\leq j<\frac{r}{2}$, 
 injective maps 
\[\phi^+_j:\ \mu(L^+,j)\ra \mu(L',j),\  \phi^-_j:\ \mu(L^-,j)\ra \mu(L',j)\setminus \im(\phi^+_j).\]
%****************************Seite 11***********************************
Such a choice is possible because $L_j/L_{j+1}$ and $L'_j/L'_ {j+1}$ have the same dimension over $\kappa_D$.
We define 
\[I'^+:=\bigcup_{0\leq j<\frac{r}{2}}(\im(\phi^+_j)\cup\sigma(\im(\phi^-_j))),\]
and we put $I'^-$ to be the complement of $I'^+$ in $I.$
Because
\beq\label{eqDualityCriteriaI}
i\in \mu(L',j)\text{ if and only if } \sigma(i)\in \mu(L',-j-1),
\eeq
for all $i\in I$, we have that $I'^+\cap \sigma(I'^+)$ is empty, and by symmetry 
$I'^-\cap \sigma(I'^-)$ is empty too. Thus 
\[\sigma(I'^+)=I'^-.\]
This new decomposition of $I$ defines   
\[W^{'+}:=\oplus_{i\in I'^+}W_i,\ W^{'-}:=\oplus_{i\in I'^-}W_i,\ L'^+:=L'\cap W^{'+} ,\ L'^-:=L'\cap W^{'-}.\]
By construction, $L'^+$ and $L^+$ are lattice sequences such that $L'^+_i/L'^+_{i+1}$ and  $L^+_i/L^+_{i+1}$ have the same $\kappa_D$-dimension. We choose an isomorphism $u$ of 
 $D$-vector spaces from $W^+$ to $W^{'+}$ such that $uL^+$ is equal to $L'^+$.
The map 
\[g:=(u,0)+\sigma((u^{-1},0)):W^+\oplus W^-\ra W^{'+}\oplus W^{'-}\]
is an element of $G$ and $gL$ is equal to $L'.$ 

\textbf{Case 2} ($L_0^{\#}=L_0$ and $r$ is odd): 
We can construct $W^{'+}$ and $W^{'-}$ as in Case 1, but we have to change the definition of $\phi^+_{\frac{r-1}{2}}$.
Because of (\ref{eqDualityCriteriaI}), the set $\mu(L',\frac{r-1}{2})$ is invariant 
under the action of $\sigma$, i.e. we can choose $\phi^+_{\frac{r-1}{2}}$ such that 
\[\sigma(\im(\phi^+_{\frac{r-1}{2}}))\cap \im(\phi^+_{\frac{r-1}{2}})=\emptyset,\]
because $L_{\frac{r-1}{2}}/L_{\frac{r+1}{2}}$ and $L'_{\frac{r-1}{2}}/L'_{\frac{r+1}{2}}$ have the same $\kappa_D$-dimension.
We now conclude as in Case 1.

\textbf{Case 3} ($L_0^{\#}=L_1$): Unlike the cases before we have
\beq\label{eqDualityCriteriaII}
i\in \mu(L',j)\text{ if and only if } \sigma(i)\in \mu(L',-j),
\eeq
We follow the proof of the cases 1 and 2, but with the following differences:
\be
\item We consider $0\leq j\leq \frac{r}{2}$, i.e. if $r$ is even the index $\frac{r}{2}$ is considered too in all formulas.
\item The set $\mu(L',0)$ is $\sigma$-equivariant. 
\item If $r$ is even the set $\mu(L',\frac{r}{2})$ is $\sigma$-equivariant.
\ee
For the $\sigma$-equivariant sets we apply the procedure of Case 2 for the choice
of the map $\phi_j^+.$ After these preparations we conclude as in Case 1 to finish the proof.
\ei
\end{proof}

We now come to the proof of Theorem \ref{thmMainclassicalEmbeddings}.

%***********************Seite 12****************************************
\begin{proposition}\label{propMainThmPart2}
Let $\phi_1$ and $\phi_2$ be two $k$-algebra monomorphisms from $E$ to $\End_D(V)$ and $x$ be a point of 
$\building(G)\cap\building_{red}(\tilde{G})^{\phi_1(E)^\times}\cap\building_{red}(\tilde{G})^{\phi_2(E)^\times}$. 
Assume that 
\be
\item $\Gcd{f(E|k)}{d}$ divides $e(E|k)$ or 
\item $\frac{m\Gcd{[E:k]}{d}}{[E:k]}$ is odd and $(\phi_1(E),\mf{a}_x)$ and $(\phi_2(E),\mf{a}_x)$ are equivalent embeddings.
\ee
Then there is an element $g$ of the stabilizer of $x$ in $G$ such that
\[\phi_1(x)=g\phi_2(x)g^{-1},\text{ for all }x\in E, \]
if and only if $(V,h^{\phi_1})$ is isomorphic to $(V,h^{\phi_2})$ as signed Hermitian $\tens{E}{k}{D}$-modules. 
\end{proposition}

\begin{proof}
We only have to prove the "if"-part. The other direction is obvious.
We write $E_1$ for $\phi_1(E)$. Let $g'$ be an isomorphism from $(V,h^{\phi_2})$ to $(V,h^{\phi_1})$. Then, $g'$ is an element of $G$ and the points $x$ and $g'.x$ are elements of $\building(G)\cap\building_{red}(\tilde{G})^{E_1^\times}$ of the same type in  $\building(G)$. 
We apply Proposition \ref{propTyppreserving} on $x$ and $g'.x$ in Case 1. and Lemma \ref{lemTypePreservingembeddingType} on $j_{E_1}(x)$ and $j_{E_1}(g'.x)$ in Case 2. We 
conclude that $j_{E_1}(x)$ and $j_{E_1}(g'.x)$ have the same type in $\building_{red}(\tilde{G}_{E_1})$.  
There is an element $g''$ of $G_{E_1}$ such that $g''g'.x=x$ by Proposition \ref{propSimplicialTypeInBGandBredtildeG}. The element $g:=g''g'$ satisfies the desired assertion. 
\end{proof}

As a condition for part two of Theorem \ref{thmMainclassicalEmbeddings} we need a new property.

\begin{definition}\label{defExtensionProperty}
A field extension of non-Archimedean local fields $E|k$ has the {\it extension property with respect to $D$} if the Galois generator $\tau$ of $E_D|k$ can be extended to an
element of $\Aut(E|k)$, the group of $k$-algebra automorphisms of $E$.
\end{definition}

\begin{proposition}\label{propMainThmPart1}
Let $\phi_1$ and $\phi_2$ be two $k$-algebra monomorphisms from $E$ to $\End_D(V)$ and $x$ be a point of 
$\building(G)\cap\building_{red}(\tilde{G})^{\phi_1(E)^\times}\cap\building_{red}(\tilde{G})^{\phi_2(E)^\times}$ such that $(\phi_1(E),\mf{a}_x)$ and $(\phi_2(E),\mf{a}_x)$
are equivalent embeddings. Assume further that $\sigma'$ is non-trivial on $E$ and that $E|k$ has the extension property with respect to $D$. Then there is an element $g$ of the stabilizer of $x$ in $G$ such that
\[\phi_1(E)=g\phi_2(E)g^{-1},\]
 if and only if $(V,h^{\phi_1})$ is isomorphic to $(V,h^{\phi_2})$ as signed Hermitian $\tens{E}{k}{D}$-modules.
\end{proposition}

For the proof of Proposition \ref{propMainThmPart1} , we will need the following Lemma:

\begin{lemma}\label{lemSpecialLambda}
Assume that $e(E|k)$ is odd and $\Gcd{f(E|k)}{d}$ is even. 
The extension $\sigma''$ of the non-trivial Galois element $\tau$ of $E_D|k$ to $E|k$ can be chosen to commute with $\sigma'$ and additionally in a way such that
there is a map $\lambda''$ which satisfies (\ref{assLambda}) and is $\sigma''$-invariant, i.e. $\lambda''\circ\sigma''$ is $\lambda''$.
\end{lemma}

\begin{proof}
We denote the residue characteristic of $k$ by $p$ and the maximal unramified field extension of $E|k$ by $E_{ur}$. 
Let $\psi$ be an extension of $\tau$ to $E|k$. The group $<\sigma'>$ acts via conjugation on 
$\psi\Aut(E|E_{ur})$ whose cardinality is odd by assumption. This action has a fixed point which we denote by $\sigma''$. We can assume that 
its order is a power of $2$, because otherwise we take an appropriate odd power of $\sigma''$. Recall that $E_0$ is the fixed field of $\sigma'$ and denote 
the fixed field of $\sigma''$ in $E_0$ by $E_{00}$. Take a map $\lambda''_0$ from $E_{00}$ to $k$ which satisfies (\ref{assLambda}) if we substitute $E$ by $E_{00}$. Then, 
\[\lambda'':=\lambda''_0\circ\trace_{E|E_{00}}\]
has the desired properties because $E|E_{00}$ is separable and tamely ramified. 
\end{proof}

The conclusion of the lemma holds in many more cases, but we do not need this generality for our purposes, since for even ramification index we can use Proposition \ref{propMainThmPart2}.

%**********************Seite 13********************************************
\begin{proof}[of \ref{propMainThmPart1}]
We only need to consider the case that $d$ is two, $e(E|k)$ is odd and $f(E|k)$ and $\frac{m\Gcd{[E:k]}{d}}{[E:k]}$ are even, because all the others are covered by Proposition \ref{propMainThmPart2}.
Every element of $G$ which stabilizes a point of the facet associated to $\mf{a}_x$ fixes the whole facet because $G$ acts type-preservingly on $\building_{red}(\tilde{G})$. Thus we can assume $x$ to be the 
barycenter of $\mf{a}_x$ in $\building_{red}(\tilde{G})$. Fix an isomorphism $g_1$ from $(V,h^{\phi_2})$ to $(V,h^{\phi_1})$. Define $y$ to be $g_1x$. In the remaining part we only work with $(V,h^{\phi_1})$ and $\phi_1(E)$, so without loss of generality let us assume that $\phi_1$ is the identity.
%In the first step we consider the case where $E|k$ is an unramified field extension of degree two. 
We consider the description of $j_E$ given in 
Proposition \ref{propChoiceOfW}, i.e. as a map 
\[j_{E}:\ \Latt{o_D}{V}^{E^\times}\ra \Latt{o_{E}}{W},\]
which has a precise description in terms of self-dual lattice functions  for its restriction to $\building(G)\cap\building_{red}(\tilde{G})^{E^\times}$.
%Let $\Lambda_x$ and $\Lambda_y$ be the $o_D$-lattice function corresponding to $x$ and $y$ respectively.
Let $\Theta_x$ and $\Theta_y$ be the self-dual lattice functions corresponding to $j_{E}(x)$ and $j_{E}(y)$ respectively. The condition on the embedding type forces the existence of an element $g'$ of $\tilde{G}_{E}$ which sends $x$ to $y$ ($x$ and $y$ are barycenters!), i.e. there is a real number $s$ such that $\Theta_y+s$ is equal to $g'.\Theta_x$. We take $g'$ such that $s$ is a minimal non-negative real number with the latter property. Such a minimum exists because if not then the set of points of discontinuity of $\Theta_y$ is dense in $\mathbb{R}$.  
We are done if $s$ is zero by Proposition \ref{propSimplicialTypeInBGandBredtildeG}. Thus assume $s$ to be positive. We have 
that $\sigma(g')g'\Theta_x$ is equal to $\Theta_x+2s$ and the minimality of $s$ shows that there is no positive $t$ smaller than $2s$ together with an element $g''$ of $\tilde{G}_E$ such that $g''\Theta_x$ is equal to $\Theta_x+t$. Thus $2s$ divides $1$, i.e. there is an integer $z$ such that $z2s=1$.

Case $2|z$: Then $2s$ divides $\frac12$ implying that the quotients $\Theta_x(t)/\Theta_x(t+)$ and $\Theta_x(t+\frac12)/\Theta_x((t+\frac12)+)$ are 
$\kappa_k$-isomorphic. The formula in (c) of Proposition \ref{propBLII31} together with Proposition \ref{propSimplicialTypeInBGandBredtildeG} shows that the type of $x$ and $y$ 
in $\building(G)$ determines the type of $j_{E}(x)$ and $j_{E}(y)$ in $\building(G_E)$, respectively, i.e. $j_{E}(x)$ and $j_{E}(y)$ have the same type because $x$ and $y$ have, which
leads to a contradiction to the positivity of $s$.

Case $z$ is odd: The fact that $zs$ is $\frac12$ implies that $\Theta_x(t)/\Theta_x(t+)$ is $\kappa_k$-isomorphic to 
\[(\Theta_y(t-\frac12)\pi_D)/(\Theta_y((t-\frac12)+)\pi_D).\]
Here we take $\lambda''$ of Lemma \ref{lemSpecialLambda} and we take the description of $j_{E}$ from Proposition \ref{propChoiceOfW} with $h^{\id_{E}}$ constructed 
from $\lambda''\circ\phi_1^{-1}$ instead of $\lambda\circ\phi_1^{-1}$. In that proposition and afterwards we construct signed Hermitian spaces $(W,h_{E})$ and $(W\pi_D,h_{E}')$ whose isometry groups are isomorphic to $G_E$. Now, by Proposition \ref{propIsomorphicCentralizerSpaces}  $(W,h_{E})$ is $E$-isomorphic to $(W,h_{E}')$, because $\sigma'$ is non-trivial on $E$ and the $E$-dimension of $W$, which is $\frac{m\Gcd{[E:k]}{d}}{[E:k]}$, is even. Thus, by Proposition \ref{propSimplicialTypeInBGandBredtildeG} there is an isomorphism from $h_E$ to $h_E'$ which maps $\Theta_x$ to $(\Theta_y+\frac12)\pi_D$ and the result follows now from Corollary \ref{corNormalizer}.
\end{proof}

The extension property of Proposition \ref{propMainThmPart1} is necessary, more precisely

\begin{proposition}\label{propExtCondNecessary}
Suppose we are given two embeddings $(E_i,\mf{a})$, $i=1,2$ , with the same self-dual hereditary order and $\sigma$-invariant fields such that 
$E_1$ is conjugate to $E_2$ by an element $\mf{a}\cap G$. 
Then at least one of the following two assertions is true:
\be
\item For every $g\in G$ wich conjugates $E_1$ to $E_2$ there is a $u\in \mf{a}\cap G$ such that for all $x\in E_1$ we have
\[gxg^{-1}=uxu^{-1}.\]
\item There is an automorphism of $E_1|k$ whose restriction to $(E_1)_D$ is non-trivial, in particular $d=2$. 
\ee
\end{proposition}
 
\begin{proof}
Assume that there are $g\in G$ and $u\in G\cap\mf{a}$ such that both conjugate $E_1$ to $E_2$ such that the conjugation of $g$ on $E_1$ can not be witnessed by an element 
of $G\cap\mf{a}$. Then $u^{-1}g$ normalizes $E_1$ and it does not fix $(E_1)_D$. Because if it would fix $(E_1)_D$ then we can apply Proposition \ref{propMainThmPart2}
on $(E_i,\mf{b}),\ i=1,2$, and $G_{E_D}$ where $\mf{b}$ is the centralizer of $E_D$ in $\mf{a}$, i.e. we could verify the conjugation of $u^{-1}g$ on $E_1$ by an element of 
$\mf{a}\cap G$. This is a contradition. 
\end{proof}

% We now give a weaker summary of Theorem \ref{thmMainclassicalEmbeddings}
% 
% \begin{proposition}
% Suppose we are given two eqiuvalent embeddings $(E_i,\mf{a})$, $i=1,2$ , with the same self-dual hereditary order and $\sigma$-invariant fields such that 
% $E_1$ is conjugate to $E_2$ by an element $G$. Assume further one of the following conditions:
% \bi
% \item[(a)] $d|e(E|k)$
% \item[(b)]  
% \ei 
% 
% \end{proposition}

%**************************************Seite 14****************************

\section{The extension property}\label{secExtensionProperty}
In this section we want to analyze the etension property of Proposition \ref{propMainThmPart1}.
For that we firstly study a stronger condition.
\subsection{The strong extension property}
We fix a finite field extension $E|k$. Let $e$ be the ramification index and $f$ be the residue class
degree of $E|k$.
 We say that $E|k$ satisfies the {\it strong extension property} if the Frobenius automorphism of the maximal unramified subextension of $E|k$ can be extended to an automorphism of $E|k$. Results concerning the strong extension property were motivated by a communication with P. Schneider. We denote by $E_{ur}|k$ the maximal unramified field extension of $E|k$. 

\begin{proposition}\label{propStrongExtensionProperty}
\bi
\item[(a)]
If $E|k$ is normal or isomorphic to a tensor product of a purely ramified and an unramified extension then $E|k$ satisfies the strong extension condition
\item[(b)]
Let $f$ be prime to $e$. The following assertions are equivalent:
\be
\item $E|k$ satisfies the strong exension property. 
\item $f|\#\Aut(E|k)$.
\item $\Aut(E|k)$ is isomorphic to a semidirect product $\Aut(E|E_{ur})\rtimes \bbZ/(f\bbZ)$. 
\item There is subgroup of $\Aut(E|k)$ of order $f$. 
\item There is a subfield extension $E'|k$ of $E|k$ of degree $e$. 
\item $E$ is $k$-algebra isomorphic to a tensor product of a purely ramified and an unramified exteinsion over $k$. 
\ee
\ei
\end{proposition}

\begin{proof}
We only need to concentrate on (b).
We have immediatly: 
\[3.\Lra 4.\Leqa 5. \Leqa 6. \Lra 2.\]
Let us denote $\Aut(E|E_{ur})$ by $N$. It is  group of order dividing $e$.  

\textbf{2. implies 1.:} The order of $N$ is prime to $f$. Thus the restriction homomorphism from $\Aut(E|k)$ to $\Gal(E_{ur}|k)$ is surjective by 2. And we deduce 1.

\textbf{1. implies 3.:} Take the $e$th power of a lift of the Frobenius automorphism to get an element of order $f$ in $\Aut(E|k)$. So we have a cyclic subgroup $C$ of order 
$f$. The restriction to $E_{ur}$ gives an isomorphism from $C$ to $\Gal(E_{ur}|k)$ because both have the same order which is prime to the order of $N$.
Thus $\Aut(E|k)$ is isomorphic to $N\rtimes C$. 
\end{proof}

We now want to decide whether or not a tamely ramified extension is a tensor product of a purely ramified and  an unramified extension over $k$. 
So let $E|k$ be tamely ramified with $f$ and $e$ not nessecarely coprime.
Define $\tilde{e}:=\Gcd{\frac{\#\kappa_E-1}{\#\kappa-1}}{e}$. We denote by $\mu_{i}(E)$ the set of $i$th roots of unity in $E$. 
\[\mu_{tame}(E|k):=\{x\in\mu_{\#\kappa_E-1}(E)\mid \exists\pi_E\in\mf{p}_E\setminus\mf{p}^2_E:\ \pi_E^e x^{-1}\in o_k\}. \]
Further we define an equivalence relation on $\mu_{\#\kappa_E-1}(E)$ by: $x\sim_{E|k} y$ if and only if $(\frac{x}{y})^{\frac{\#\kappa_E-1}{\tilde{e}}}=1$.
Let $\zeta$ be a primitive $\tilde{e}$th root of unity in $E$. Then an easy exercise shows that $\sim_{E|k}$ is under 
\[\psi:\bbZ/(\#\kappa_E-1)\bbZ\ra \mu_{\#\kappa_E-1}(E),\ \bar{1}\mapsto\zeta\]
the push forward of the equivalence relation given by $\tilde{e}\bbZ/(\#\kappa_E-1)\bbZ$. 

%********************************Seite 15********************************
\begin{proposition}\label{propTamelyTensorProd}
Let $E|k$ be tamely ramified. Then $\mu_{tame}(E|k)$ is non-empty and an equivalence class of $\sim_{E|k}$ and the following conditions are equivalent.
\be
\item $E|k$ is a tensor product of a purely ramified and an unramified field extension over $k$. 
\item $E|k$ contains an $e$th root of some uniformizer of $k$. 
\item $1\in\mu_{tame}(E|k)$.
\item There exists uniformizers $\pi_E$ and $\pi_k$ for $E$ and $k$, respectively, such that 
$(\frac{\pi_E^e}{\pi_k})^{\frac{\#\kappa_E-1}{\tilde{e}}}$ is an element of $ 1+\mf{p}_E$
\item For all uniformizers $\pi_E$ and $\pi_k$ for $E$ and $k$, respectively, the element 
$(\frac{\pi_E^e}{\pi_k})^{\frac{\#\kappa_E-1}{\tilde{e}}}$ lies in $ 1+\mf{p}_E$
\ee
\end{proposition}

The concreteness of the proposition allows to construct tamely ramified examples where the strong extension condition fails.
It also implies an easy corollary which is a genralization of \cite[II.5.12]{lang:94}. The latter states that if $f$ is equal to $1$ there is some uniformizer of $k$ which has an $e$th root in $E$.

\begin{corollary}\label{corRoots}
Let $E|k$ be a field extension. Suppose that $e'$ is coprime to 
\[\Char(\kappa)\sum_{i=0}^{f-1}(\#\kappa)^i.\] 
Then there is a uniformizer $\pi_k$ of $k$ and an element $\alpha$ of $E$ such that $\alpha^{e'}$ is equal to $\pi_k$.
\end{corollary}

For example the assumption of Corollary \ref{corRoots} is satisfied if $e'$ is 
a power of $2$ and $f\Char(\kappa)$ is odd. 

\begin{proof}
Without loss of generality we assume that $E|k$ is tamely ramified because we can turn to the maximal tamely ramified subextension of $E|k$. 
We apply \cite[II.5.12]{lang:94} to find an $e'$th root $x$ of a uniformizer of $E_{ur}$ in $E$. We can now apply Proposition \ref{propTamelyTensorProd}
on $E_{ur}[x]|k$ because $e'$ is prime to $\sum_{i=0}^{f-1}(\#\kappa)^i$ which is $\frac{\#\kappa_E-1}{\#\kappa-1}$.
\end{proof}

\begin{proof}[Prop. \ref{propTamelyTensorProd}] The non-emptiness of $\mu_{tame}(E|k)$ comes from \cite[II.5.12]{lang:94} and the surjectivity of the $e$th power map from $1+\mf{p}_{E_{ur}}$ to $1+\mf{p}_{E_{ur}}$ . 
For the next assertion let $x$ be an element of $\mu_{tame}(E|k)$. Then $y\in\mu_{\#\kappa_E-1}(E)$ is equivalent to $x$ if and only if $x=uy$ for some $u\in\mu_{\frac{\#\kappa_E-1}{\tilde{e}}}(E)$
if and only if $x=u'vy$ for some $u'=u^e,\ u\in\mu_{\#\kappa_E-1}(E)$, and $v\in o_k^\times$. The ''only if''-part is a consequence of Bezout's lemma: 
Let $\zeta$ be a primitive $\#\kappa_E-1$-root of unitiy. By the Henselian lemma: $\zeta\in E_{ur}$. Then $u$ is an $\tilde{e}j$ power of $\zeta$ for some integer $j$. Now Bezout's lemma
states that there are integers $a$ and $b$ such that 
\[a e+b\frac{\#\kappa_E-1}{\#\kappa-1}=\tilde{e}.\]
Now the $\frac{\#\kappa_E-1}{\#\kappa-1}$th power of $\zeta$  is an element of $k$ which proves the ''only if'' part. 
The equivalence of the first three assertions is either trivial or part of \cite[II.5.12]{lang:94}. We have also $5.\Lra 4.$ and $3.\Lra 4.$ From 4. follows 5. because every element $u$ of $o_E^\times$ satisfies $u^{\frac{e(\#\kappa_E-1)}{\tilde{e}}}\in 1+\mf{p}_E$.  We finish the proof in showing that 5. implies 3. Take an element $x$ of $\mu_{tame}(E|k)$. 5. states that $x$ is equivalent to 1 and thus we get 3.
\end{proof}

\subsection{The extension property}
We now combine the statements of Proposition \ref{propTamelyTensorProd} and 
Theorem \ref{thmMainclassicalEmbeddings} to get a statement when two appropriate field extensions are conjugate under $G$. 

\begin{remark}
For a field extension $E|k$ in $\End_D(V)$ the extension property with respect to $D$ is equivalent to the strong extension property of $E|E'$ where $E'|k$ is the maximal unramified subextension of $E|k$ of degree prime to $d$.
\end{remark}

%***************************Seite 16******************************
\begin{theorem}
Assume that $E_1|k$ and $E_2|k$ are two $\sigma$-invariant but not $\sigma$-fixed field extensions in $\End_D(V)$ such that $E_i^\times$  is contained in the normalizer of a self-dual hereditary order $\mf{a}$ and such that there are $g_1\in\mf{n}(\mf{a})$ and $g_2\in G$ which conjugate $E_1$ to $E_2$. Suppose that there is no element of $\mf{a}\cap G$ which conjugates $E_1$ to $E_2$. Then $e(E_1|k)$ is odd and $f(E_1|k)$ and $d$ are even and $E_1|E'$ does not satisfy the strong extension property where $E'$ is the maximal unramified subextension of $E_1|k$ of odd degree. More precisely
\bi
\item[(i)] None of the statements from 1. to 6. in Proposition \ref{propStrongExtensionProperty} hold for $E_1|E'$.
\item[(ii)] None of statements from 1. to 5. in Proposition \ref{propTamelyTensorProd} hold for $E_1|E'$ if $E_1|k$ is tamely ramified.
\ei
\end{theorem}

\begin{appendix}
\section{Embedding of buildings of classical groups}
This appendix is devoted to Proposition \ref{propJEClassocalGroup}. Its proof can be found in part one of the proof of Lemma 3.27 in \cite{skodlerack:10}.
The existence of a signed Hermitian form for $G_E$ can be found in \cite[1.12]{skodlerack:10}. In the latter dissertation the author did not mention the explicit formula for the restriction of 
$j_E$ to $\building(G)\cap \building_{red}(\tilde{G})^{E^\times}$. For this, one needs to construct a signed Hermitian form $h_E$ whose unitary group is $G_E$. 
The rest is stated in \cite{broussousLemaire:02}. We assume the situation of section \ref{secExtHermForms} with 
the restriction that $\phi$ is $\id_E$. In this section we use extensively that the residue characteristic of $k$ is different from $2$. We denote by  $\Lattone{o_D}{V}^{E^\times}$ the set these $o_D$-lattice functions of $V$ which are $o_E$-lattice functions.
Let us recall how $j_E$ is constructed in our situation.

\begin{proposition}[\cite{broussousLemaire:02},~II.3.1]\label{propBLII31}
There is a pair $(\Delta,W)$ consisting of a skew-field $\Delta$ which is $\tens{E}{k}{D}$ or $E$ and a $\Delta$-subvector space $W$ of $V$ such that 
\be
\item $\End_{\tens{E}{k}{D}}(V)$ is $E$-algebra isomorphic to $\End_{\Delta}(W)$ via 
\beq\label{eqIsoCentr} 
a\mapsto a|_W,
\eeq
and 
\item the map $j_E$ from $\building_{red}(\tilde{G})^{E^\times}$ to $\building_{red}(\tilde{G}_E)$ in terms of lattice functions has the form 
\[j_E([\Lambda])=[\Lambda\cap W].\]
\ee

In more detail:
\bi 
\item[(a)] We have that $\Delta$ is $\tens{E}{k}{D}$ if $[E:k]$ is odd and $E$ otherwise. 
Secondly $V/W$ is $\Gcd{d}{[E:k]}$-dimensional over $\Delta$. 
\item[(b)] Let $L|k$ be a maximal unramified field extension in $D$. If $\Gcd{d}{[E:k]}=2$ and $f(E|k)$ is odd then $V$ is $\tens{E}{k}{D}$-isomorphic to $\tens{W}{k}{L}$ and 
$\Lambda\in\Lattone{o_D}{V}^{E^\times}$ is equal to $\tens{(\Lambda\cap W)}{o_k}{o_L}$.
\item[(c)] If $\Gcd{d}{f(E|k)}=2$ then $V$ is $\tens{E}{k}{D}$-isomorphic to $W\oplus W\pi_D$ where $\pi_D$ is a uniformizer of $D$ and 
$\Lambda\in\Lattone{o_D}{V}^{E^\times}$ is equal to 
\[(\Lambda\cap W)\oplus((\Lambda\cap W)+\frac12)\pi_D.\]
\ei
\end{proposition}

The point of finding $h_E$ is the right choice of $W$. The proof of Proposition \ref{propBLII31} is also included in the proof of:

\begin{proposition}\label{propChoiceOfW}
We can choose $W$ in Proposition (\ref{propBLII31}) in a way such that 
\bi
\item the image $\mf{M}$ of $h^{\id_E}|_{W\times W}$ is a one dimensional bi-$\Delta$-vector space, and 
\item there is a bi-$\Delta$-isomorphism $\psi$ from $\mf{M}$ to 
$\Delta$ such that $h_E:=\psi\circ h^{\id_E}|_{W\times W}$ satisfies the following.
\be 
\item The map (\ref{eqIsoCentr}) is $\sigma$-$\sigma_{h_E}$-equivariant. In particular 
$G_E$ is isomorphic to $\U(h_E)$ via (\ref{eqIsoCentr}). 
\item For an element $\Lambda$ of $\Lattone{o_D}{V}^{E^\times}$ we have:
\[\Lambda^{\#}\cap W=\left\{\begin{array}{rl}
                             (\Lambda\cap W)^{\#,h_E}+\frac12,&\text{if }\Gcd{d}{f(E|k)}=2\text{ and }\trace(\mf{M})=\{0\}  \\
			     (\Lambda\cap W)^{\#,h_E},&\text{otherwise}\\			     
                            \end{array}\right.
\]
In the first case $(\Lambda\cap W)-\frac14$ is self-dual if $\Lambda$ is.
\ee
\ei    
In particular in terms of self-dual lattice functions $j_E|_{\building(G)\cap\building_{red}(\tilde{G})^{E^\times}}$ has the form
\[\Lambda\mapsto \left\{\begin{array}{rl}
                             (\Lambda\cap W)-\frac14,&\text{if }\Gcd{d}{f(E|k)}=2\text{ and }\trace(\mf{M})=\{0\}  \\
			     (\Lambda\cap W),&\text{otherwise}\\			     
                            \end{array}\right.\]
\end{proposition}

%****************************Seite 17************************************
\begin{proof}
\textbf{\underline{Case 1:}} ($\Gcd{d}{[E:k]}=1$) Then $W$ is $V$ and $\Delta$ is $D$ and we apply Proposition \ref{prop5_5brousStevForDneqkSkewField}.

\textbf{\underline{Case 2:}} ($\Gcd{d}{[E:k]}=2$ and $f(E|k)$ is odd.)  Here two divides $e(E|k)$ and we can establish the following situation.
There is a two-dimensional unramified and $\rho$-invariant field
 extension $L$ in $D,$ and an element $p_E$ of $E$ such that $p_E^2,$ which we denote by $\pi_k,$ is a prime
 element of $k$, see Corollary \ref{corRoots}. Further we find a square root $\pi_D$ of $\pi_k$ in $D$, which normalizes $L$, such that 
 $\rho(\pi_D)$ is either $+\pi_D$ of $-\pi_D$. We denote the non-trivial element of $\Gal(L|k)$ by $\tau.$
 Let $W$ be an arbitrary $\frac{\dim_E(V)}{2}$-dimensional $E$-vector space. 
 We define a right-$D$-action on $\tens{W}{k}{L}$ via
 \[(\tens{w}{k}{l}).\pi_D:=(\tens{p_E w}{k}{l}).\]
The space $\tens{W}{k}{L}$ is $\tens{E}{k}{D}$-module isomorphic to $V$ by the theory of semi-simple modules and we identify them.
 
\underline{Step 2.1:} The algebra $\tens{E}{k}{D}$ is $E$-algebra isomorphic to $\Matr_2(E)$ via
 \[\pi_D\mapsto \left(\begin{matrix}
                       p_E & 0\\
                       0 & -p_E\\
                      \end{matrix}\right), 
 l'\mapsto \left(\begin{matrix}
                       0 & l'^2\\
                       1 & 0\\
                      \end{matrix}\right), 
 \]
 where $l'$ is a unit in $L$ which satisfies $\tau(l')=-l'.$ We identify $\tens{E}{k}{D}$ with $\Matr_2(E)$.
 Let $\mf{M}$ be the image of $W\times W$ under $h^{\id_E}$ . For elements $A$ of $\mf{M}$ we have
 \bi
 \item $A p_E=A\pi_D$ and
 \item $\rho(\pi_D)A=\sigma`(p_E)A,$
 \ei
 especially 
 if 
 \beq\label{eqSymCase}
 \rho(\pi_D)\sigma`(p_E)=\pi_Dp_E
 \eeq
 we have
 \[\mf{M}=\left\{\left(\begin{matrix}
                        e & 0\\
                        0 & 0
                       \end{matrix}\right)\mid e\in E\right\},
 \]
 and
 if
 \beq\label{eqSkewCase}
 \rho(\pi_D)\sigma`(p_E)=-\pi_Dp_E
 \eeq
 we have
 \[\mf{M}=\left\{\left(\begin{matrix}
                        0 & 0\\
                        e & 0
                       \end{matrix}\right)\mid e\in E\right\}.
 \]
 We define now a signed hermitian form $h_E$ on $W$ to be $\psi\circ h^{\id_E}|_{W\times W}$. Here $\psi$ is $a_{11}|_\mf{M}$ for (\ref{eqSymCase}) and $a_{21}|_\mf{M}$ for 
(\ref{eqSkewCase}) where $a_{ij}$ is the map from $\Matr_2(E)$ to $E$ which maps a matrix to its entry on position $(i,j).$

%************************************Seite 18***************************************
 \underline{Step 2.2:} Here we prove 
\[(\Lambda\cap W)^{h_E}=\Lambda^{\#}\cap W,\]
i.e. we have to prove that $h_E(w,M)$ is a subset of $\mf{p}_E$ if and only if $h(w,\tens{M}{o_k}{o_L})$ is a subset of $\mf{p}_D$,
 for all $w\in W $ and all full $o_E$-lattices $M$ of $W.$ It follows from 
 \[\psi^{-1}(\mf{p}_E)=\mf{M}\cap\bigcap_{u\in o_E\setminus\{0\}}u^{-1}\tilde{\lambda}^{-1}(\mf{p}_D).\]
 To prove the last equation we use the decomposition
 \beq\label{eqETensorD}
\Matr_2(E)=E\oplus El'\oplus E\pi_D \oplus El'\pi_D.
\eeq For example for Case (\ref{eqSymCase}):
 Let $A$ be an element of $\mf{M}$ with coefficients $e_1,e_2,e_3,e_4$ in the above decomposition, then we have \[e_2=e_4=0,\ e_1-e_3p_E=0,\]
 and every such matrix is an element of $\mf{M}$. Thus  $\psi(A)\in \mf{p}_E$ if and only if 
 $e_1\in \mf{p}_E$,
 i.e. if and only if 
 \[\lambda(e_1o_E)\subseteq \mf{p}_k\text{ and } \lambda(e_3o_E)\subseteq o_k,\]
 i.e. if and only if
 \[\tilde{\lambda}(o_EA)\subseteq \mf{p}_D.\]
The Case (\ref{eqSkewCase}) is analogous with the equations 
\[e_1=e_3=0,\ e_2-e_4p_E=0.\]

\textbf{\underline{Case 3:}} ($\Gcd{d}{f(E|k)}=2$) We follow a similar strategy as in Case 2. We fix an unramified two-dimensional field-extension $L$ of $k$ in $D$ which is $\rho$-invariant and a prime element $\pi_D$ of $D$ which normalizes $L$, such that 
the square of $\pi_D$ is a prime element of $k$ denoted by $\pi_k$ and $\rho(\pi_D)$ is equal to $+\pi_D$ or $-\pi_D$. We can embed $L$ into $E$, because $2|f(E|k)$, and we identify $L$ with its image under the embedding.
As in Case 2 we identify $\tens{E}{k}{D}$ with $\Matr_2(E)$ but now via 
\[\pi_D\mapsto \left(\begin{matrix}
                      0 & \pi_k\\
                      1 & 0\\
                     \end{matrix}\right), 
l'\mapsto \left(\begin{matrix}
                      l' & 0\\
                      0 & -l'\\
                     \end{matrix}\right). 
\]

\underline{Step 3.1:} We consider the idempotents $1^1$ and $1^2$ in $\tens{E}{k}{L}$, i.e. $1^1$ is the matrix $E_{1,1}$ and $1^2$ is the matrix $E_{2,2}$ in the standard notation of linear algebra.
We define $V^i:=1^iV$ and $W:=V^1$. Recall that $V^2=V^1\pi_D.$
Let $\mf{M}$ be the image of $W\times W$ under $h^{\id_E}$. The equations
\bi
\item $A 1^1=A$ and
\item $(\tens{\sigma'}{k}{\rho})(1^1)A=A,$
\ei
for $A\in \mf{M}$ 
imply that 
\beq\label{eqImageDiag}
\mf{M}=\left\{\left(\begin{matrix}
                       e & 0\\
                       0 & 0
                      \end{matrix}\right)\mid e\in E\right\},
\eeq
or 
\beq\label{eqImageAntiDiag}
\mf{M}=\left\{\left(\begin{matrix}
                       0 & 0\\
                       e & 0
                      \end{matrix}\right)\mid e\in E\right\}.
\eeq
We take the same $\psi$ as in Case 2 to define $h_E.$

%****************************Seite 19********************************
\underline{Step 3.2:} Consider the lattice function $\Lambda$ under the decomposition $V=V^1\oplus V^2$. We define $(\Theta_\Lambda)(t)$ to be $V^1\cap \Lambda(t)$ for (\ref{eqImageDiag}) and $V^1\cap \Lambda(t+\frac14)$ for (\ref{eqImageAntiDiag}), i.e. $\Lambda(t)$ (resp. $\Lambda(t+\frac14)$) has the form 
\beq\label{eqBLdis2}\Theta_\Lambda(t)\oplus\Theta_\Lambda(t-\frac12)\pi_D.\eeq 
Here $\Theta_\Lambda$ is an element of $\Lattone{o_E}{W}$. 
We have to prove:
\[(\Theta_\Lambda)^{\#,h_E}=\Theta_{\Lambda^{\#}}.\]
\underline{Step 3.2a:} At first, we show the equivalence of the following two statements for an element $A$ of $\mf{M}.$
\be
\item $\psi(A)\in \mf{p}_E.$
\item $\tilde{\lambda}(o_E A)\subseteq \mf{p}_D^2.$
\ee
For this we look at decomposition (\ref{eqETensorD}) of $\tens{E}{k}{D}$. and we obtain for $A$ the relations
\be
\item[(\ref{eqImageDiag})] $e_3=e_4=0$ and $e_1=l'e_2.$
\item[(\ref{eqImageAntiDiag})] $e_1=e_2=0$ and $e_3=-l'e_4.$
\ee
From these relations the equivalences follow from (\ref{assLambda}).
\vspace{1em}\\
\underline{Step 3.2b:} We only consider (\ref{eqImageAntiDiag}). The other case is similar. We have to show 
\[\Theta_\Lambda^{\#,h_E}(t)=\Lambda^{\#}(t+\frac14)\cap W.\]
\vspace{1em}\\
$\supseteq:$ This follows directly from $2.\Ra 1.$. 
\vspace{1em}\\
$\subseteq:$ Let $w$ be an element of $\Theta_\Lambda^{\#,h_E}(t)$, i.e. $h(w,\Theta_\Lambda((-t)+))$ is a subset of $\mf{p}_D^2$, and 
 more precisely it is a subset of $\mf{p}_D^3$ since it is contained in $k\mf{p}_D.$
Thus
\[h(w,1^1\Lambda((-t+\frac14)+))=h(w,\Theta_\Lambda((-t)+))\subseteq \mf{p}_D^3,\]
i.e.
\[h(w,1^2\Lambda((-t-\frac14)+)\subseteq \mf{p}^2_D \text{ and }h(w,1^1\Lambda((-t-\frac34)+)\subseteq \mf{p}_D \]
and especially
\[h(w,\Lambda((-t-\frac14)+))\subseteq \mf{p}_D.\]
Thus $w$ is an element of $\Lambda^{\#}(t+\frac14)$.
\end{proof}

The last proposition answers how to find an element of $G$ which centralizes $E$. But sometimes we look for an element of the 
normalizer of $E$ in $G$ which is not in the centralizer. This is what the last part of the section is about. 

\textbf{For this assume that we are in Case 3 of the proof above, which means $\Gcd{f(E|k)}{d}=2$}. Define
$W':=W\pi_D$, 
\beq
h'_E:=\left\{\begin{array}{rl}
             a_{22}\circ h^{\id_E}|_{W'\times W'},&\text{ if (\ref{eqImageDiag})}\\
	     \pi_k^{-1}a_{12}\circ h^{\id_E}|_{W'\times W'},&\text{ if (\ref{eqImageAntiDiag})}\\
            \end{array}\right.
\eeq 
and
\beq
\Theta'_{\Lambda}(t):=\left\{\begin{array}{rl}
             \Lambda(t)\cap W',&\text{ if (\ref{eqImageDiag})}\\
	     \Lambda(t-\frac14)\cap W',&\text{ if (\ref{eqImageAntiDiag})}\\
            \end{array}\right.,
\eeq
for $\Lambda\in\Lattone{o_D}{V}^{E^\times}$.
Analogous to Proposition \ref{propChoiceOfW} we have the equality $\Theta'^{\#,h'_E}_\Lambda=\Theta'_{\Lambda^\#}$ and affine isomorphisms 
\[\building(\U(h_E))\cong\building(G)\cap\building_{red}(\tilde{G})^{E^\times}\cong\building(\U(h'_E)),\ \Theta_\Lambda\mapsto\Lambda\mapsto\Theta'_\Lambda.\]

%*****************************************Seite 20*********************************
\begin{corollary}\label{corNormalizer}
Suppose we are given two elements $\Lambda_1$ and $\Lambda_2$ of $\building(G)\cap\building_{red}(\tilde{G})^{E^\times}$ and
an isomorphism from $(W,h_E)$ to $(W',h'_E)$ which maps $\Theta_{\Lambda_1}$ to $\Theta'_{\Lambda_2}$. Then for every $\sigma''\in\Aut(E|k)$
 which 
\bi 
\item commutes with $\sigma'$ and 
\item which has nontrivial restriction to $E_D$ and 
\item for which $\lambda\circ\sigma''$ is $\lambda$
\ei
there is an element $g$ of $G$ such that $g\Lambda_1$ is equal to $\Lambda_2$ and $geg^{-1}$ is equal to $\sigma''(e)$ for all elements $e$ of $E$.
\end{corollary}

\begin{proof}
 Consider instead of $h'_E$ the form $\sigma''\circ h'_E$ and $W'$ with the structure 
\[e*v:=\sigma''^{-1}(e)v,\ e\in E,\ v\in W'.\]
If we look at Gram matrices we see immediately that both signed Hermitian spaces $(W',h'_E)$ and $(W',\sigma''\circ h'_E)$ are isomorphic. 
Take a splitting basis $(w'_i)_i$ of $\Theta'_{\Lambda_2}$ in $W'$ corresponding to a Witt decomposition of $W'$ with respect to $h'_E$. Then there is an isomorphism from $(W',h'_E)$ to $(W',\sigma''\circ h'_E)$ which maps $w'_i$ 
into $o_E^\times *w'_i$ for all $i$.  Thus this isomorphism fixes  $\Theta'_{\Lambda_2}$, i.e. if we use our assumption we get that there is an isomorphism $f$ from $(W,h_E)$ to $(W',\sigma''\circ h'_E)$ 
such that $\Theta_{\Lambda_1}$ is mapped to $\Theta'_{\Lambda_2}$. From the definition of $h_E$ and $h'_E$ made of $h^{\id_E}$ and the fact that 
$\lambda\circ\sigma''$ is equal to $\lambda$ we deduce 
\[h(f(v),f(w))=h(v,w),\]
for all $v,w\in W$. We extend $f$ via
\[f(v+w\pi_D):=f(v)+f(w)\pi_D,\ v,w\in W.\]
This $f$ is $D$-linear, because for $w,v\in W$ we have
\begin{eqnarray*}
f((v+w\pi_D) l')&=&f(l'v-l'(w\pi_D))\\
&=&f(l'v)-f(l'w)\pi_D\\
&=&l'*f(v)-(l'*f(w))\pi_D\\
&=&-l'f(v)+(l'f(w))\pi_D\\
&=&f(v)l'-f(w)l'\pi_D\\
&=&(f(v)+f(w)\pi_D)l'\\
&=&(f(v+w\pi_D))l'.\\
\end{eqnarray*}
\end{proof}

It is not clear in general that $(W,h_E)$ is isomorphic to $(W',h'_E)$. But:

\begin{proposition}\label{propIsomorphicCentralizerSpaces}
The signed Hermitian space $(W,h_E)$ is always isomorphic to $(W,\pi_k h'_E)$. 
There is an isomorphism from $(W,h_E)$ to $(W,h'_E)$, i.e. anisotropic parts are isomorphic, if:
\be
\item $\dim_EW$ is even and the anisotropic parts have not dimension 2 or
\item $\dim_EW$ is even and $\sigma'$ does not fix $E$ pointwise or
\item $\pi_k$ is a square in $E$.
\ee
\end{proposition}

\begin{proof} 
Firstly assume that $\rho(\pi_D)$ is $\pi_D$. 
We define $f$ from $W$ to $W'$ via 
\[f(w):=w\pi_D^{-1},\ w\in W.\]
Then we have for (\ref{eqImageAntiDiag})
\begin{eqnarray*}
 \pi_k h'_E(f(v),f(w)) &=&  a_{21}(h^{\id_E}(f(v),f(w)))\\
&=& a_{21}(\rho(\pi_D^{-1}) h^{\id_E}(v,w)\pi_D^{-1})\\
&=& a_{21}(\pi_D^{-1} h^{\id_E}(v,w)\pi_D^{-1})\\
&=& a_{12}(h^{\id_E}(v,w))\\
&=& h_E(v,w)\\
\end{eqnarray*}
and for (\ref{eqImageDiag})
\begin{eqnarray*}
 \pi_k h'_E(f(v),f(w)) &=&  \pi_k a_{22}(h^{\id_E}(f(v),f(w)))\\
&=& a_{22}(\pi_D h^{\id_E}(v,w)\pi_D^{-1})\\
&=& a_{11}(h^{\id_E}(v,w))\\
&=& h_E(v,w).\\
\end{eqnarray*}
If $\rho(\pi_D)$ is $-\pi_D$ we get that $h_E$ is isomorphic to $-\pi_k h'_E$. The latter is isomorphic to $\pi_k h'_E$ because $-1$ is a square 
in $E_D|k$. 
Thus we have that an anisotropic part of $h_E$ has the same $E$-dimension as one of $h'_E$. Recall that such a dimension is not greater than $4$. 
If $\pi_k$ is a square in $E$ then it is also a norm with respect to $\Nm_{E|E_0}$ because $-1$ is a square and the residue characteristic is odd. Thus $\pi_k h'_E$ is isomorphic to $h'_E$. The other cases are made the way such that the dimension of the anisotropic part determines its isomorphism class. 
\end{proof}

\end{appendix}

\bibliographystyle{alpha}
%\bibliography{/home/zhoudani/EigeneDateien/ArbeitStudiumForschung/LaTeX/bib/bibliography}
\bibliography{../../LaTeX/bib/bibliography}

\end{document}

%% file: HeaderDSkJan2011.tex
\usepackage{theorem}
\usepackage{amsmath}
\usepackage{amsfonts}
\usepackage{amssymb}
\usepackage[all]{xy}
\entrymodifiers={!!<0pt,.7ex>+}

\newtheorem{theorem}{Theorem}[section]
\newenvironment{proof}{\textbf{Proof:} }{\text{q.e.d.}}
\newtheorem{assumption}[theorem]{Assumption}
\newtheorem{proposition}[theorem]{Proposition}
{\theorembodyfont{\rmfamily}
\newtheorem{notation}[theorem]{Notation}

\newtheorem{definition}[theorem]{Definition}

\newtheorem{lemma}[theorem]{Lemma}
\newtheorem{corollary}[theorem]{Corollary}
\newtheorem{remark}[theorem]{Remark}

\newtheorem{example}[theorem]{Example}}

\newcommand{\ncmd}{\newcommand}
\ncmd{\be}{\begin{enumerate}}
\ncmd{\ee}{\end{enumerate}}
\ncmd{\bi}{\begin{itemize}}
\ncmd{\ei}{\end{itemize}}
\ncmd{\beq}{\begin{equation}}
\ncmd{\eeq}{\end{equation}}
\ncmd{\bdfn}{\begin{definition}}
\ncmd{\edfn}{\end{definition}}
\ncmd{\bprop}{\begin{proposition}}
\ncmd{\eprop}{\end{proposition}}
\ncmd{\bthm}{\begin{theorem}}
\ncmd{\ethm}{\end{theorem}}
\ncmd{\brem}{\begin{remark}}
\ncmd{\erem}{\end{remark}}
\ncmd{\bcor}{\begin{corollary}}
\ncmd{\ecor}{\end{corollary}}
\ncmd{\bex}{\begin{example}}
\ncmd{\eex}{\end{example}}
\ncmd{\bnot}{\begin{notation}}
\ncmd{\enot}{\end{notation}}
\ncmd{\bass}{\begin{assumption}}
\ncmd{\eass}{\end{assumption}}
\ncmd{\blem}{\begin{lemma}}
\ncmd{\elem}{\end{lemma}}
\ncmd{\bproof}{\begin{proof}}
\ncmd{\eproof}{\end{proof}}

\ncmd{\idxS}[1]{\sindex[indexsymb]{#1}}
\ncmd{\idxD}[1]{\sindex[indexdef]{#1}}

\ncmd{\lra}{\mbox{$\longrightarrow$}}
\ncmd{\Lra}{\mbox{$\Longrightarrow$}}
\ncmd{\ura}{\underrightarrow}
\ncmd{\ua}{\mbox{$\uparrow$}}
\ncmd{\da}{\mbox{$\downarrow$}}
\ncmd{\hra}{\mbox{$\hookrightarrow$}}
\ncmd{\ra}{\mbox{$\rightarrow$}}
\ncmd{\Ra}{\mbox{$\Rightarrow$}}
\ncmd{\eqa}{\mbox{$\leftrightarrow$}}
\ncmd{\Eqa}{\mbox{$\Leftrightarrow$}}
\ncmd{\leqa}{\mbox{$\longleftrightarrow$}}
\ncmd{\Leqa}{\mbox{$\Longleftrightarrow$}}
\ncmd{\iso}{\mbox{$\tilde{\ra}$}}
\ncmd{\liso}{\mbox{$\tilde{\lra}$}}%\ncmd{\liso}{\mbox{$\stackrel{\sim}{\rigtharrow}$}}

\ncmd{\Reinfty}{\mbox{$\mathbb{R}\cup\{\infty\}$}}
\ncmd{\bbN}{\mathbb{N}}
\ncmd{\bbZ}{\mbox{$\mathbb{Z}$}}
\ncmd{\bbR}{\mbox{$\mathbb{R}$}}
\ncmd{\ggT}{\operatorname{ggT}}
\ncmd{\Gcd}[2]{\mathop{\text{gcd}(#1,#2)}}

\ncmd{\mf}[1]{\mathfrak{#1}}
\ncmd{\calli}[1]{\mathcal{#1}}
\ncmd{\Rcal}{\calli{R}}
\ncmd{\Scal}{\calli{S}}
\ncmd{\gfrak}{\mf{g}}
\ncmd{\hfrak}{\mf{h}}
\ncmd{\ti}[1]{\tilde{#1}}
   \ncmd{\tiV}{\ti{V}}\ncmd{\tiD}{\ti{D}}\ncmd{\tisigma}{\ti{\sigma}}
   \ncmd{\tij}{\ti{j}}\ncmd{\tii}{\ti{i}}
   \ncmd{\titfh}{\ti{\mf{h}}}\ncmd{\titfgt}{\ti{\mf{g}}}
   \ncmd{\tiJ}{\ti{J}}\ncmd{\tiJplus}{\ti{J}^+}\ncmd{\tiG}{\ti{G}}
   \ncmd{\tibG}{\mathbf{\ti{G}}} \ncmd{\tibH}{\mathbf{\ti{H}}}
\ncmd{\leftexp}[2]{{\vphantom{#2}}^{#1}{#2}}

\ncmd{\graded}[2]{\mbox{$\mf{gr}_{#1}(#2):=\oplus_i({#1}^i{#2}/{#1}^{i+1}{#2})$}}

\ncmd{\umin}[1]{\underset{#1}{\min}}
\ncmd{\uinf}[1]{\underset{#1}{\inf}}
\ncmd{\usup}[1]{\underset{#1}{\sup}}
\ncmd{\umax}[1]{\underset{#1}{\max}}
\ncmd{\uomin}[2]{\underset{#1}{\overset{#2}{\min}}}
\ncmd{\uoinf}[2]{\underset{#1}{\overset{#2}{\inf}}}
\ncmd{\uosup}[2]{\underset{#1}{\overset{#2}{\sup}}}
\ncmd{\uomax}[2]{\underset{#1}{\overset{#2}{\max}}}

\ncmd{\sequence}[1]{\mbox{$({#1}_n)_\mathbb{N}$}}

\ncmd{\pF}{\mf{p}_F}
\ncmd{\oK}{\mbox{$o_K$}}
\ncmd{\pK}{\mbox{$\mf{p}_K$}}
\ncmd{\pD}{\mbox{$\mf{p}_D$}}
\ncmd{\pFpow}[1]{\mbox{$\mf{p}_F^{#1}$}}
\ncmd{\pDpow}[1]{\mbox{$\mf{p}_D^{#1}$}}
\ncmd{\oDelta}{\mbox{$o_{\Delta}$}}
\ncmd{\pDelta}{\mbox{$\mf{p}_{\Delta}$}}
\ncmd{\vr}[1]{\mbox{$o_{#1}$}}
\ncmd{\vi}[1]{\mbox{$\mf{p}_{#1}$}}
\ncmd{\vipower}[2]{\mbox{$\mf{p}_{#1}^{#2}$}}
\ncmd{\Char}{\operatorname{char}}

\ncmd{\Building}{\mf{B}^1}
\ncmd{\building}{\mf{B}}
\ncmd{\Ical}{\mathcal{I}}
\ncmd{\IAE}{\mbox{$\mathcal{I}^{E^{\times}}$}}
\ncmd{\IB}{\mbox{$\mathcal{I}_E$}}
\ncmd{\Om}[1]{\mbox{$\Omega_{#1}$}}
\ncmd{\OmE}[1]{\mbox{$\Omega_{E,{#1}}$}}
\ncmd{\bary}{\mathop{bary}}
\ncmd{\typ}{\operatorname{typ}}
\ncmd{\Stern}{\operatorname{Stern}}
\ncmd{\Fix}{\operatorname{Fix}}
\ncmd{\diam}{\operatorname{diam}}

\ncmd{\Norm}[2]{\operatorname{Norm}_{#1}(#2)}
\ncmd{\Normone}[2]{\operatorname{Norm}^1_{#1}(#2)}
\ncmd{\Normtwo}[2]{\operatorname{Norm}^2_{#1}(#2)}

\ncmd{\ord}{\oparatorname{ord}}
\ncmd{\Ord}{\operatorname{Ord}}
\ncmd{\lattices}[1]{\operatorname{lattices}({#1})}
\ncmd{\rad}[1]{\operatorname{rad}({#1})}
\ncmd{\Her}[1]{\operatorname{Her}({#1})}
\ncmd{\Latt}[2]{\operatorname{Latt}_{#1}(#2)}
\ncmd{\Gitter}[2]{\operatorname{Latt}(#2,#1)}
\ncmd{\LC}{\operatorname{LC}}
\ncmd{\Lattone}[2]{\operatorname{Latt}^1_{#1}(#2)}
\ncmd{\Latttwo}[2]{\operatorname{Latt}^2_{#1}(#2)}

\ncmd{\Inn}{\operatorname{Inn}}
\ncmd{\Centr}{\operatorname{Z}}%\ncmd{\Centr}[2]{\mathop{\text{C}_{#1}({#2})}}
\ncmd{\Z}{\operatorname{Z}}
\ncmd{\lmult}{\operatorname{l}}
\ncmd{\rmult}{\operatorname{r}}
\ncmd{\ind}{\operatorname{ind}}
\ncmd{\Gal}{\mathop{\text{Gal}}}
\ncmd{\stab}{\mathop{\text{stab}}}
\ncmd{\LietiG}{\mbox{$\mathop{Lie}(\tilde{G})$}}
\ncmd{\LF}{\operatorname{LF}}%Liealgebrafiltration
\ncmd{\gyr}{\mf{g}_{x,r}}
\ncmd{\hxr}{\mf{h}_{x,r}}

\ncmd{\Matr}{\operatorname{M}}%\ncmd{\Matr}[2]{\mathop{\text{M}_{#1}({#2})}}
\ncmd{\trd}{\operatorname{trd}}
\ncmd{\Nrd}{\operatorname{Nrd}}
\ncmd{\Nm}{\operatorname{N}}
\ncmd{\trace}{\operatorname{tr}}
\ncmd{\Gram}{\operatorname{Gram}}
\ncmd{\gr}{\operatorname{\mf{gr}}}
\ncmd{\degree}{\operatorname{deg}}
\ncmd{\Skew}{\operatorname{Skew}}
\ncmd{\Sym}{\operatorname{Sym}}
\ncmd{\diag}{\operatorname{diag}}
\ncmd{\antidiag}{\operatorname{antidiag}}

\ncmd{\EMatr}[1]{\mathds{1}_{#1}}
\ncmd{\Mint}[3]{\mathop{\text{M}_{#1,#2}(#3)}}
\ncmd{\trans}[1]{{#1}^{\ensuremath{\mathsf{T}}}} 

\ncmd{\Lie}{\operatorname{Lie}}
\ncmd{\Abb}{\operatorname{Abb}}
\ncmd{\Top}{\operatorname{Top}}
\ncmd{\Hom}{\operatorname{Hom}}
\ncmd{\End}{\operatorname{End}}
\ncmd{\Aut}{\operatorname{Aut}}%\ncmd{\Aut}[2]{\mathop{\text{Aut}_{#1}(#2)}}
\ncmd{\Iso}{\operatorname{Iso}}

\ncmd{\im}{\operatorname{im}}
\ncmd{\rang}{\operatorname{rang}}

\ncmd{\id}{\operatorname{id}}

\ncmd{\Span}{\operatorname{span}}%\ncmd{\Span}[2]{\mathop{\text{span}_{#1}({#2})}}
\ncmd{\tens}[3]{{#1}\otimes_{#2}{#3}}
\ncmd{\proj}{\operatorname{proj}}

\ncmd{\PO}{\operatorname{PO}}
\ncmd{\SL}{\text{SL}}
\ncmd{\GL}{\text{GL}}
\ncmd{\Sp}{\operatorname{Sp}}
\ncmd{\SO}{\operatorname{SO}}
\ncmd{\Ogp}{\operatorname{O}}
\ncmd{\Rad}{\operatorname{Rad}}
\ncmd{\nr}{\operatorname{nr}}
\ncmd{\GLDV}{\mbox{$\GL_DV$}}
\ncmd{\SLDV}{\mbox{$\SL_DV$}}
\ncmd{\U}{\operatorname{U}}
\ncmd{\SU}{\operatorname{SU}}
\ncmd{\AfSp}[2]{\operatorname{\mathbf{A}}^{#1}_{K}}
\ncmd{\Res}{\operatorname{Res}}
\ncmd{\bV}{\operatorname{\mathbf{V}}}
\ncmd{\bA}{\operatorname{\mathbf{A}}}
\ncmd{\bGm}{\operatorname{\mathbf{G}_{\mathbf{m}}}}
\ncmd{\bG}{\operatorname{\mathbf{G}}}
\ncmd{\bH}{\operatorname{\mathbf{H}}}
\ncmd{\bMatr}{\operatorname{\mathbf{M}}}
\ncmd{\bU}{\operatorname{\mathbf{U}}}
\ncmd{\bSU}{\operatorname{\mathbf{SU}}}
\ncmd{\bSL}{\operatorname{\mathbf{SL}}}
\ncmd{\bO}{\operatorname{\mathbf{O}}}
\ncmd{\bSO}{\operatorname{\mathbf{SO}}}
\ncmd{\bSp}{\operatorname{\mathbf{Sp}}}
\ncmd{\bGL}{\operatorname{\mathbf{GL}}}
\ncmd{\bT}{\operatorname{\mathbf{T}}}
\ncmd{\bAfSp}[1]{\operatorname{\mathbf{A}^{#1}}}
\ncmd{\Weylgroup}{\mbox{$\leftexp{v}{W}$}}
\ncmd{\Osheaf}[1]{\mbox{$\mf{O}_{#1}$}}
\ncmd{\RingedSpace}[1]{\mbox{$(#1,\Osheaf{#1})$}}
\ncmd{\maximalIdeal}[1]{\mbox{$\mf{m}_{#1}$}}

\ncmd{\Invariante}[2]{\mathop{\text{Iv}({#1},{#2})}}
\ncmd{\pairs}{\mathop{\text{pairs}}}

\ncmd{\zzmatrix}[4]{\left( \begin{array}{cc}#1&#2\\#3&#4 \end{array}\right)}

\ncmd{\Eb}{Euclidean building}
\ncmd{\Gr}{geometric realisation}
\ncmd{\vs}{vector space}
\ncmd{\nd}{non-degenerated}
\ncmd{\na}{non-archimedean}
\ncmd{\fd}{finite dimensional}